\documentclass[11pt,leqno]{amsart}
\usepackage{amssymb}
\usepackage{amsmath, amsfonts,enumerate}
\usepackage{hyperref}
\usepackage{amsthm}
\usepackage{latexsym,bm}
\usepackage{euscript}

\let\cal\mathcal

\makeatletter \@mparswitchfalse \makeatother
\newtheorem{theorem}{Theorem}
\newtheorem{lemma}[theorem]{Lemma}

\newtheorem{corollary}[theorem]{Corollary}
\newtheorem{proposition}[theorem]{Proposition}
\theoremstyle{remark}
\newtheorem{remark}[theorem]{Remark}

\theoremstyle{definition}
\newtheorem{definition}[theorem]{Definition}

\theoremstyle{remark}

\numberwithin{equation}{section}
\numberwithin{theorem}{section}

\def\Z{\mathbb Z}

\def\R{\mathbb R}
\def\M{\cal{M}}
\def\H{\cal{H}}
\let\<\langle
\def\ch{\raise 0.5ex \hbox{$\chi$}}
\def\T{\tau}
\def\E{\cal{E}}

\let\\\cr
\let\phi\varphi

\let\epsilon\varepsilon

\def\dim{\operatorname{dim}}

\renewcommand{\a}{\alpha}
\renewcommand{\b}{\beta}
\newcommand{\g}{\gamma}

\newcommand{\e}{\varepsilon}
\newcommand{\8}{\infty}
\newcommand{\el}{\ell}

\renewcommand{\l}{\lambda}
\newcommand{\La}{\Lambda}

\newcommand{\s}{\sigma}

\newcommand{\h}{\mathsf{h}}
\newcommand{\at}{\mathrm{at}}
\newcommand{\alg}{\mathrm{aa}}
\newcommand{\cat}{\mathrm{crude}}

\newcommand{\be}{\begin{eqnarray*}}
\newcommand{\ee}{\end{eqnarray*}}
\newcommand{\beq}{\begin{equation}}
\newcommand{\eeq}{\end{equation}}

\begin{document}

\title{ Atomic decompositions for noncommutative martingales}

\author[Chen]{Zeqian Chen}
\address{Wuhan Institute of Physics and Mathematics, Chinese Academy of Sciences, Wuhan 430071, China}
 \email{zqchen@wipm.ac.cn}

\author[Randrianantoanina]{Narcisse Randrianantoanina}
\address{Department of Mathematics, Miami University, Oxford,
Ohio 45056, USA}
 \email{randrin@miamioh.edu}

\author[Xu]{Quanhua Xu}
\address {Institute for Advanced Study in Mathematics, Harbin Institute of Technology,  Harbin 150001, China; and Laboratoire de Math{\'e}matiques, Universit{\'e} de Bourgogne Franche-Comt{\'e}, 25030 Besan\c{c}on Cedex, France}
\email{qxu@univ-fcomte.fr}

\date{}

\begin{abstract}
We prove an atomic type decomposition for the noncommutative  martingale Hardy space $\h_p$ for all $0<p<2$ by an explicit constructive method using algebraic atoms as building blocks. Using this elementary construction, we obtain a weak form of the atomic decomposition of $\h_p$ for all $0< p < 1,$ and provide a constructive proof of the atomic decomposition for $p=1$. We also study $(p,\8)_c$-atoms, and show that every $(p,2)_c$-atom can be decomposed into a sum of $(p,\8)_c$-atoms; consequently, for every $0<p\le 1$, the $(p,q)_c$-atoms lead to the same atomic space for all $2\le q\le\8$. As applications, we obtain a characterization of the dual space of  the noncommutative martingale Hardy space $\h_p$ ($0<p<1$) as a noncommutative Lipschitz space via the weak form of the atomic decomposition. Our constructive method can also be applied to proving  some sharp martingale inequalities.

\end{abstract}
\maketitle

\section{introduction}

This paper follows the current line of investigation on noncommutative martingale inequalities. Thanks to its interactions with other fields such as operator spaces, noncommutative harmonic analysis and free probability, the theory of noncommutative martingale inequalities has been steadily  developing since the establishment of the noncommutative Burkholder-Gundy inequalities in \cite{PX}. In return, these noncommutative inequalities have important applications to operator spaces and quantum stochastic analysis. See for instance, \cite{Ju-OH, ju-par-GAFA, JX3, jx-orlicz, pis-shly, Xu-G} for some illustrations of applications to operator space theory.
 Many classical results have been successfully transferred to the noncommutative setting. One of them, directly relevant to the subject of the present paper, is the so-called   atomic decompositions for the noncommutative martingale Hardy spaces $\H_1$  and $\h_1$ in \cite{Bekjan-Chen-Perrin-Y}.   Atomic  decompositions are fundamental in the classical martingale theory and harmonic analysis. For instance, they are powerful tools in dealing with  various aspects  of martingale Hardy spaces such as duality, interpolation, and  many others.  Contrary to the commutative case, the approach to these decompositions in \cite{Bekjan-Chen-Perrin-Y} is based on duality arguments and therefore  not constructive. This difficulty is  explained by the noncommutativity of operator product and the lack of an efficient analogue of the notion of stopping times. Since then it had been an open problem to find a constructive proof for  the atomic decomposition of \cite{Bekjan-Chen-Perrin-Y}.  An important motivation of  finding such a constructive approach  is that it would provide new insights on Hardy spaces $\H_p$ and $\h_p$  for all $0<p<1$ that have  been previously left untouched since  duality arguments are no longer available for this range.  We would like to emphasize that Hardy  spaces for $0<p<1$ are also important objects in the classical theory. For instance, atomic  decompositions for the classical $\h_p$ for $0<p<1$ were   also extensively studied  (cf. e.g., \cite{Weisz2,Weisz}). It was also an open problem to obtain  atomic decompositions for the noncommutative  Hardy space $\h_p$ for $0<p<1$.

\smallskip

The present paper sheds  light  on all the problems mentioned above and provides a much clearer picture of the current state of arts concerning the atomic decomposition of noncommutative Hardy spaces. 

\smallskip

We give a constructive proof of the atomic decomposition for the noncommutative  Hardy space $\h_1$, thus solve the main open problem of \cite{Bekjan-Chen-Perrin-Y}. We also obtain, through an explicit method,  an atomic decomposition for the noncommutative $\h_p$ for all $0<p<1$ when the so-called algebraic atoms are used. Using this elementary construction, we obtain a weak form of the atomic decomposition of $\h_p$ for all $0< p < 1.$ The latter result allows us to describe the dual space of the quasi-Banach space $\h_p$ (for $0<p<1$) as noncommutative Lipschitz space which was another problem left open in \cite{Bekjan-Chen-Perrin-Y}.  The notion of algebraic atoms for noncommutative martingales  first appeared in the thesis of  Perrin  \cite{Perrin2}. More recently, algebraic atomic Hardy spaces  were  extensively used for the study of noncommutative maximal functions  in \cite{Hong-Junge-Parcet}. In these early instances,  only Hardy  spaces in the
Banach space range $1\leq p<2$ were considered. We formulate the notion of algebraic atoms in the more general contexts of column/row conditioned Hardy spaces to the range $0<p<1$ and use these as building blocks of  algebraic atomic Hardy spaces in this range.

\smallskip

Our approach in Section~\ref{Atomic decomposition}   is very different from the commutative case. Surprisingly,   our constructions  do   not make use  of Cuculescu's projections which  are  very often necessary  as  substitute  of  classical stoping times in the noncommutative  setting.  

\smallskip

In the definition of the atoms mentioned previously, one uses the $L_2$-norm, the resulting atoms are the so-called $(p,2)_c$-atoms. Motivated by the classical theory, Hong and Mei \cite{Hong-Mei} introduced  $(1,q)_c$-atoms for any $1<q<\8$, using the $L_q$-norm instead of the $L_2$-norm, and showed that these $(1,q)_c$-atoms lead to the same atomic space. However, their method does not work for $q=\8$ while the $(1,\8)$-atoms are the commonly used and nicest atoms in the commutative setting. Another important aspect of the present paper is to make up for this deficiency. We define $(p,\8)_c$-atoms for any $0<p\le1$ and show that these atoms lead to the same atomic space $\h_{p,\at}^{c}$ as the $(p,2)_c$-atoms. Contrary to the approach of \cite{Hong-Mei} which is based on duality, ours is constructive and explicitly decomposes every $(p,2)_c$-atom into $(p,\8)_c$-atoms. Based on Cuculescu's projections, this proof is quite elaborate and technical.

\medskip

The paper is organized as follows. In the next section,   we collect  notions and notation from noncommutative martingale theory necessary for the whole paper.  Section~\ref{Atomic decomposition} is devoted to the atomic decomposition of the noncommutative Hardy spaces $\h_p$ where $0<p<2$. The building blocks of the atomic spaces considered  are the $(p, 2)$-atoms introduced in \cite{Bekjan-Chen-Perrin-Y} and  their variants known as $(p,2)$-crude atoms formulated in \cite{Hong-Mei}. The most important notion being used however is the algebraic $\h_p^c$-atoms. In Section \ref{(p,infty)-atom}, we consider the case of $(p,q)$-atoms and prove that all atomic Hardy spaces in terms of $(p, q)$-atoms coincide for $2 \le q \le \8.$ In particular, we prove that $\h_{1,\at}^{c} = \h_{1, \8}^{c}$ which improves the related results in \cite{Bekjan-Chen-Perrin-Y} and \cite{Hong-Mei}. In Section~\ref{Applications},  we provide two important applications of our results from Sections~\ref{Atomic decomposition} and \ref{(p,infty)-atom}.


\section{Preliminary definitions}


\subsection{Noncommutative spaces}
Throughout this paper, $\M$ will always denote a von
Neumann algebra with a normal faithful normalized finite trace $\T$.  The unit of $\M$ will be denoted by ${\bf 1}$. For each
$0 < p \leq \infty,$ let $L_p (\M, \T)$ (or simply $L_p
(\M)$) be the  noncommutative $L_p$-space associated with the pair $(\M,\T)$. We refer to \cite{PX3} for details and  more historical references on noncommutative $L_p$-spaces.

For $x \in L_p (\M)$ we denote by $r(x)$ and $l(x)$ the right and left support projections  of $x$, respectively.
Recall that if $x = u |x|$ is the polar decomposition of $x$, then $r(x) = u^* u$ and $l(x) = u u^*$.  The projection
$r(x)$ (resp. $l(x)$) can  be also  characterized as the least projection $e$ such that $x e =x$ (resp. $e x =x$).
If $x$ is selfadjoint, then $r(x) = l(x)$; in this case we  simply call it the support projection of $x$ and  denote it by $s(x)$.  If $x\in L_p(\M)$ is  selfadjoint
 and $x = \int^\infty_{- \infty} s d e^x_s$ is
its spectral decomposition, then for any Borel subset $B \subseteq
\R$, we denote by $\ch_B(x)$ the corresponding spectral projection
$\int^\infty_{- \infty} \ch_B(s) d e^x_s$.

We now record few lemmas for further use.  The first one is an  elementary observation; it is a particular case  of the noncommutative Minkowski type inequality for  the case of matrices (\cite{Carlen-Lieb1999}). We include a proof for the convenience of the reader.

\begin{lemma}\label{rev-triangle}
Let $0<r<1$. Then for any positive $a$ and $b$ are positive operator in $L_r(\M)$
\[
\|a\|_r +\|b\|_r \leq \|a + b \|_r\,.
\]

\end{lemma}
\begin{proof}
 Considering $a +b +\epsilon{\bf 1}$ with $\epsilon>0$ instead of $a +b$ if necessary, we can assume $a +b$ invertible. Write
\[
\|a\|_r^r=\T \Big( \big[ (a+b)^{(r^2-r)/2} a^r(a+b)^{(r^2-r)/2}\big] (a+b)^{r-r^2}\Big)\,.
\]
Using H\"older's inequality,
\[
\|a\|_r^r\leq \T \Big( \big[ (a+b)^{(r^2-r)/2} a^r(a+b)^{(r^2-r)/2}\big]^{1/r}\Big)^{r} \T \Big( (a+b)^{(r-r^2)/(1-r)}\Big)^{1-r}\,.
\]
Since $1/r >1$, we apply \cite{Kosaki} to further get  that
\begin{equation*}
\begin{split}
\|a\|_r 
&\leq \T \big(  (a+b)^{(r-1)/2} a(a+b)^{(r-1)/2}\big) \T \big( (a+b)^{(r-r^2)/(1-r)}\big)^{(1-r)/r}\\
&=\T \big(  (a+b)^{r-1} a\big) \T \big( (a+b)^{r}\big)^{(1-r)/r}\,.
\end{split}
\end{equation*}
Similarly,
\[
\|b\|_r \leq\T \big(  (a+b)^{r-1} b\big) \T \big( (a+b)^{r}\big)^{(1-r)/r}\,.
\]
Thus,
\begin{equation*}
\begin{split}
\|a\|_r +\|b\|_r &\leq \T \big(  (a+b)^{r-1}(a + b)\big) \T \big( (a+b)^{r}\big)^{(1-r)/r}
=\T \big( (a+b)^{r}\big)^{1/r}\,.
\end{split}
\end{equation*}
The lemma is proved.
\end{proof}

As an immediate consequence of Lemma~\ref{rev-triangle}, we have  the following inequality:
\begin{lemma}\label{com-l2}
Let $0<p \leq 2$. Then, for any sequence $(a_n)_{n\geq 1}$ in $L_p(\M)$,
\[
\Big(\sum_{n\geq 1} \big\| a_n\big\|_p^2 \Big)^{1/2} \leq  \Big\| \big(\sum_{n\geq 1}|a_n|^2 \big)^{1/2}\Big\|_p.
\]
\end{lemma}
\begin{proof}
Since the case $p=2$ is trivial, we assume that $0<p<2$.  It follows immediately from Lemma~\ref{rev-triangle} that
\begin{equation*}
\begin{split}
\sum_{n\geq 1} \big\| a_n\big\|_p^2 &= \sum_{n\geq 1} \| |a_n|^2 \|_{p/2}\\
&\leq  \| \sum_{n\geq 1} |a_n|^2 \|_{p/2}\\
&=\big\| \big(\sum_{n\geq 1} |a_n|^2\big)^{1/2}\big\|_p^2.
\end{split}
\end{equation*}
This  verifies the desired inequality.
\end{proof}

The above lemma can also be deduced from  noncommutative  Khintchine inequality
(\cite{LP4, LPI, P-Ricard}) but  with some constants. The next lemma is  implicit in the proof of \cite[Proposition~3.2]{Bekjan-Chen-Perrin-Y}.  It  can also be deduced from \cite[Lemma~7.3]{JX} but with a different constant.

\medskip

\begin{lemma}\label{mart-2}
 Let $0<p <2$ and $a$   be an invertible positive operator  with bounded inverse.  If  $0\leq b$  and $b^2 \leq a^2$, then
 \[
 \T\big( a^{-2+p} (a^2 -b^2)\big) \leq  \frac{2}{p} \T\big(a^p-b^p\big).
 \]
\end{lemma}


\subsection{Noncommutative martingales}
Let us now recall the general setup for noncommutative martingales.
In the sequel, we always denote by $(\M_n)_{n \geq 1}$ an
increasing sequence of von Neumann subalgebras of ${\M}$
whose union  is w*-dense in
$\M$. For $n\geq 1$, ${\E}_n$ denotes the trace preserving conditional expectation
from ${\M}$ onto  ${\M}_n$.

\begin{definition}
A sequence $x = (x_n)_{n\geq 1}$ in $L_1(\M)$ is called \emph{a
noncommutative martingale} with respect to $({\M}_n)_{n \geq
1}$ if $\mathcal{E}_n (x_{n+1}) = x_n$ for every $n \geq 1.$
\end{definition}
If in addition, all $x_n$'s belong to $L_p(\mathcal{M})$ for some $1
\leq p \leq \infty$  then  $x$ is called an $L_p$-martingale.
In this case we set
\begin{equation*}\| x \|_p = \sup_{n \geq 1} \|
x_n \|_p.
\end{equation*}
If $\| x \|_p < \infty$, then $x$ is called
a bounded $L_p$-martingale.

Let $x = (x_n)$ be a noncommutative martingale with respect to
$(\M_n)_{n \geq 1}$.  Define $dx_n = x_n - x_{n-1}$ for $n
\geq 1$ with the usual convention that $x_0 =0$. The sequence $dx =
(dx_n)$ is called the \emph{martingale difference sequence} of $x$. A martingale
$x$ is called  \emph{a finite martingale} if there exists $N$ such
that $d x_n = 0$ for all $n \geq N.$
In the sequel, for any operator $x\in L_1(\M)$,  we denote $x_n=\E_n(x)$ for $n\geq 1$.

Let us now  review the definitions of the square functions and Hardy spaces of noncommutative martingales.
Following \cite{PX}, we consider  the column and row versions of square functions
relative to a (finite) martingale $x = (x_n)$ as follows:
 \[
 S_{c,n} (x) = \Big ( \sum^n_{k = 1} |dx_k |^2 \Big )^{1/2}, \quad
 S_c (x) = \Big ( \sum^{\infty}_{k = 1} |dx_k |^2 \Big )^{1/2}
 \]
and
 \[S_{r,n} (x) = \Big ( \sum^n_{k = 1} | dx^*_k |^2 \Big )^{1/2}, \quad
 S_r (x) = \Big ( \sum^{\infty}_{k = 1} | dx^*_k |^2 \Big)^{1/2}.
 \]
Let $0 <p < \infty$.
Define $\mathcal{H}_p^c (\mathcal{M})$
(resp. $\mathcal{H}_p^r (\mathcal{M})$) as the completion of all
finite $L_\infty$-martingales under the (quasi) norm $\| x \|_{\mathcal{H}_p^c}=\| S_c (x) \|_p$
(resp. $\| x \|_{\mathcal{H}_p^r}=\| S_r (x) \|_p $).
The mixture Hardy space of noncommutative martingales is defined as
follows. For $0< p < 2,$
\begin{equation*}
\mathcal{H}_p(\mathcal{M})
= \mathcal{H}_p^c (\mathcal{M}) + \mathcal{H}_p^r(\mathcal{M})
\end{equation*}
equipped with the (quasi) norm
\begin{equation*}
\| x \|_{\mathcal{H}_p} =
\inf \big \{ \| y\|_{\mathcal{H}_p^c} + \| z \|_{\mathcal{H}_p^r} \big\},
\end{equation*}
where the infimum is taken over all
$y \in\mathcal{H}_p^c (\mathcal{M} )$ and $z \in \mathcal{H}_p^r(\mathcal{M} )$
such that $x = y + z$.
For $2 \leq p <\infty,$
\begin{equation*}
\mathcal{H}_p (\mathcal{M}) =
\mathcal{H}_p^c (\mathcal{M}) \cap \mathcal{H}_p^r(\mathcal{M})
\end{equation*}
equipped with the norm
\begin{equation*}
\| x \|_{\mathcal{H}_p} =
\max \big \{ \| x\|_{\mathcal{H}_p^c} ,\; \| x \|_{\mathcal{H}_p^r} \big\}.
\end{equation*}
The differences between the two cases  $0 < p < 2$ and $2 \leq p  \leq \infty$ are now well-documented in the literature.

We now consider the conditioned version of $\H_p$ developed in \cite{JX}.
Let $x = (x_n)_{n \geq 1}$ be a finite martingale in $L_2(\M)$.
We set (with the convention that $\E_0=\E_1$)
 \[
 s_{c,n} (x) = \Big ( \sum^n_{k = 1} \E_{k-1}|dx_k |^2 \Big )^{1/2}, \quad
 s_c (x) = \Big ( \sum^{\infty}_{k = 1} \E_{k-1}|dx_k |^2 \Big )^{1/2}
 \]
and
 \[
 s_{r,n} (x) = \Big ( \sum^n_{k = 1} \E_{k-1}| dx^*_k |^2 \Big )^{1/2}, \quad
 s_r (x) = \Big ( \sum^{\infty}_{k = 1} \E_{k-1}| dx^*_k |^2 \Big)^{1/2}.
 \]
These are called the column and row conditioned square functions, respectively.
Let $0< p \le \infty$.
We define $\h_p^c (\mathcal{M})$ (resp. $\h_p^r (\mathcal{M})$) as the completion of all
finite $L_\infty$-martingales under the (quasi) norm $\| x \|_{\h_p^c}=\| s_c (x) \|_p$
(resp. $\| x \|_{\h_p^r}=\| s_r (x) \|_p $). Note that for $2 \le p \le \8,$ $\h_p^c (\mathcal{M})$ (resp. $\h_p^r (\mathcal{M})$) coincides with the space of all martingales $x$ for which $s_c (x) \in L_p (\M)$ (resp. $s_r (x) \in L_p (\M)$) and $\|x\|_{\h_p^c}=\| s_c (x) \|_p$ (resp. $\|x\|_{\h_p^r}=\| s_r (x) \|_p$).

We remark that by the boundedness of the conditional expectations, we have for $1\leq p<\infty$,
\[
L_p(\M)= L_p^0(\M) \oplus L_p(\M_1)
\]
where $L_p^0(\M)= \big\{ x \in L_p(\M); \E_1(x)=0\big\}$. It is worth pointing out that such direct sum is not valid for $0<p<1$ since conditional expectations are not well-defined in this range. However,
the same phenomena occur for $0<p<1$ when Hardy spaces are used. Indeed, since $L_2(\M)$ is dense in $\h_p^c(\M)$ and  $\max\big\{ \|x-\E_1(x)\|_{\h_p^c}, \| \E_1(x)\|_p \big\} \leq \|x\|_{\h_p^c}$ for  every $x \in L_2(\M)$, we may state that
\begin{equation}\label{direct-sum}
\h_p^c(\M)= \h_p^{0,c}(\M) \oplus L_p(\M_1)
\end{equation}
where $\h_p^{0,c}(\M)$ is the completion of  the linear space $L_2^0(\M)$ under the $\h_p^c$-norm. This direct sum allows us to formally isolate the first term of any given martingale from $\h_p^c(\M)$ which will be very crucial in the sequel.

We also need $\ell_p(L_p(\M))$, the space of all sequences $a=(a_n)_{n\geq 1}$ in $L_p(\M)$ such that
 \[
 \|a\|_{\ell_p(L_p(\M))}=\Big(\sum_{n\geq 1}\|a_n\|_p^p\Big)^{1/p} <\infty.
 \]
Let $\h_p^d(\M)$ be the subspace of $\ell_p(L_p(\M))$ consisting of all martingale difference sequences.

We define the conditioned version of martingale Hardy spaces as follows.
If $0< p < 2,
$\begin{equation*}
\h_p(\mathcal{M}) = \h_p^d (\mathcal{M}) +  \h_p^c (\mathcal{M})
+  \h_p^r (\mathcal{M})
\end{equation*}
equipped with the (quasi) norm
\begin{equation*}
\| x \|_{\h_p} = \inf \big \{ \| w \|_{ \h_p^d} + \| y
\|_{ \h_p^c} + \| z \|_{ \h_p^r} \big \},
\end{equation*}
where the infimum is taken over all $w \in  \h_p^d
(\M), y \in  \h_p^c (\M)$, and $ z \in \h_p^r
(\M)$ such that $ x = w + y + z.$
 If $2 \leq p <\infty,
$\begin{equation*}\h_p (\M) =  \h_p^d
(\M) \cap  \h_p^c (\M) \cap  \h_p^r (\M)
\end{equation*}equipped with the norm\begin{equation*}
\| x \|_{\h_p} = \max \big \{ \| x \|_{ \h_p^d}, \|x
\|_{ \h_p^c}, \| x \|_{ \h_p^r} \big \}.
\end{equation*}
From the noncommutative Burkholder-Gundy inequalities and Burkholder inequalities in \cite{JX, PX}, we have for every $1<p<\infty$,
 \[
 \mathcal{H}_p (\M)= \h_p ({\M}) = L_p({\M})
 \]
with equivalent norms. In this paper, we will be mainly  concerned with  the case $0<p\leq 1$ but most of the tools we use apply to $1<p<2$ too.


\section{Atomic decomposition}\label{Atomic decomposition}


We begin with introducing various concepts of noncommutative atoms from \cite{Bekjan-Chen-Perrin-Y,Hong-Mei, Perrin2}.

\begin{definition} Let $0<p <2$.  An operator $a \in L_2(\M)$ is called a $(p,2)_c$-{\em atom}, if there exist $n\geq 1$ and a projection $e\in \M_n$ such that
\begin{enumerate}[{\rm (i)}]

\item  $\mathcal{E}_n (a) =0$;

\item  $r(a) \leq e$;

\item  $ \| a \|_{2} \leq \T(e)^{1/2-1/p}$.
\end{enumerate}
Replacing $\mathrm{(ii)}$ by $ \mathrm{(ii)'}~~l (a)
\leq e$,  we have the notion of $(p,2)_r$-\emph{atoms}.
\end{definition}

Clearly, $(p,2)_c$-atoms and $(p,2)_r$-atoms are noncommutative analogues of $(p,2)$-atoms  for commutative martingales. We refer to \cite{Weisz2, Weisz} for more on the notion of atoms in the classical setting.

\begin{definition}
 Let  $0<p\le 1$. Let $\h_{p,\at}^{c}(\M)$ be the space of all operators $x \in L_p(\M)$ which  can be represented as 
\[
x= \sum_k \lambda_k a_k \quad \text{(convergence in $L_p(\M)$)},
\]
where for each $k$, $a_k$ is either a $(p,2)_c$-atom  or an element of the unit ball of $L_p(\M_1)$, and $\lambda_k \in \mathbb{C}$ satisfying $\sum_k|\lambda_k|^p <\infty$.  For $x\in\h_{p,\at}^{c}(\M)$ we define 
\[
\|x\|_{\h_{p,\at}^c}=\inf\Big(\sum_k |\lambda_k|^p\Big)^{1/p},
\]
where the infimum is taken over all  decompositions of $x$ described above. 
 \end{definition}

\begin{remark}
 We could also extend the previous definition of $\h_{p,\at}^{c}(\M)$ to the range $1<p<2$; but then one easily sees that $\|x\|_{\h_{p,\at}^c}=0$ for any $(p,2)_c$-atom $x$. An alternate choice would be to take the $\el_1$-norm of the sequence $(\l_k)$ in the above infimum for $1<p<2$; however the resulting space is too small to coincide with $\h_{p}^{c}(\M)$. Despite these drawbacks, some of our results subsist for $1<p<2$. One of them is that any $x\in \h_{p}^{c}(\M)$ admits an atomic decomposition
 \[
x= \sum_k \lambda_k a_k \; \text{ with }\;  \sum_k|\lambda_k|^p <\infty.
\]
  \end{remark}

It is clear that $\h_{p,\at}^{c}(\M)$ is a quasi-Banach space (a Banach space for $p=1$). Similarly, we also define the row version $\h_{p,\at}^r(\M)$ using  $(p,2)_r$-atoms. For mixed Hardy spaces, we  define the atomic Hardy space by setting for $0<p\le 1$,
\[
\h_{p,\at}(\M)=\h_{p}^d(\M) + \h_{p,\at}^c(\M) + \h_{p,\at}^r(\M)
\]
equipped with the (quasi) norm
\begin{equation*}
\| x \|_{\h_{p,\at}} = \inf \big\{ \| w \|_{ \h_p^d} + \| y
\|_{ \h_{p,\at}^c} + \| z \|_{ \h_{p,\at}^r} \big\},
\end{equation*}
where the infimum is taken over all $w \in  \h_p^d
(\M), y \in  \h_{p,\at}^c (\M)$, and $ z \in \h_{p,\at}^r
(\M)$ such that $ x = w + y + z$.

\medskip

A weakening of the notion of noncommutative atoms
was   introduced in \cite{Hong-Mei} for $p=1$. We formulate it here for $0<p<2$. This weaker notion will play an important role in the sequel.

\begin{definition}
 Let $0<p< 2$. An operator $a \in L_p(\M)$ is called a $(p,2)_c$-{\em crude atom}, if there exist $n\geq 1$ and a factorization $a=yb$ such that:
\begin{enumerate}[{\rm (i)}]
\item  $y\in L_{2}(\M)$, $\E_n(y)=0$ and $\|y\|_{2} \leq 1$;
\item  $b\in L_{q}(\M_n)$ with  $\|b\|_{q} \leq 1$, where $1/p=1/2 +1/q $.
\end{enumerate}
Replacing  the factorization  above by $a=by$,  we have the notion of $(p,2)_r$-\emph{crude atoms}.
\end{definition}

We may   consider  another column atomic Hardy space based on  crude atoms as building blocks. That is, for $0<p\le1$, we define  $\h_{p,\cat}^c(\M)$ to be the space of $x \in L_p(\M)$ admitting  a column crude atomic decomposition:
\begin{equation}\label{c-atom}
x= \sum_k \lambda_k a_k \quad \text{(convergence in $L_p(\M)$)},
\end{equation}
where for each $k$, $a_k$ is a $(p,2)_c$-crude atom  or an element of the unit ball of $L_p(\M_1)$, and $\lambda_k \in \mathbb{C}$ satisfying $\sum_k|\lambda_k|^p <\infty$. $\h_{p,\cat}^c(\M)$ is equipped with
\begin{equation*}
\|x\|_{\h_{p,\cat}^c}=\inf\Big(\sum_k |\lambda_k|^p\Big)^{1/p},
\end{equation*}
where the infimum is taken over all  decompositions of $x$ as in \eqref{c-atom}.

 With obvious modifications, we may also define $\h_{p,\cat}^r(\M)$  and $\h_{p,\cat}(\M)$. 

\medskip

 We now introduce a  third  type of atomic decomposition that has been considered in the literature and is central for the  present  paper:

\begin{definition}\label{alg}
 Let $0<p<2$.
 An operator $x \in L_p(\M)$ is called an {\em algebraic $\h_p^c$-atom}, whenever it can be written in the form $x=\sum_{n\geq 1} y_nb_n$, with $a_n$ and $b_n$ satisfying the following condition for $1/p=1/2 +1/q$:
\begin{enumerate}[{\rm (i)}]

\item  $\mathcal{E}_n (y_n) =0$ and $b_n \in L_q(\M_n)$ for all $n\geq 1$;

\item  $\displaystyle{\sum_{n\geq 1} \big\|y_n\big\|_2^2 \leq 1}$ and $\displaystyle{\Big\| \Big( \sum_{n\geq 1} |b_n|^2 \Big)^{1/2}\Big\|_q \leq 1}$.
\end{enumerate}
\end{definition}

The above definition was considered in \cite{Perrin2} for the range $1\leq p<2$ and more recently  this notion was used  in \cite{Hong-Junge-Parcet}  to study maximal functions of noncommutative  martingales.

Naturally, this concept of atoms  leads to the consideration  of another Hardy space:
for $0<p<2$, we say that an operator $x \in L_p(\M)$ admits an algebraic $\h_p^c$-atomic decomposition if 
\[
x =\sum_{k} \lambda_k a_k,
\] 
where for each $k$, $a_k$ is an algebraic $\h_p^c$-atom  or an element of the unit ball of $L_p(\M_1)$, and $\lambda_k \in \mathbb{C}$ satisfying $\sum_k|\lambda_k|^p <\infty$ for $0<p \leq 1$ and 
 $\sum_k|\lambda_k|<\infty$ for $1<p<2$. The corresponding \emph{ algebraic atomic column martingale Hardy space} $\h_{p,\alg}^c(\M)$ is defined to be  the space of all $x$ which admit a algebraic $\h_p^c$-atomic decomposition and is equipped with
\[
\|x\|_{\h_{p,\alg}^c}=\inf\Big(\sum_k |\lambda_k|^p\Big)^{1/p} \quad \text{for $0<p\le1$}
\]
and
\[
\|x\|_{\h_{p,\alg}^c}=\inf\,\sum_k |\lambda_k| \quad \text{for $1< p<2$},
\]
where the infimum are   taken over all  decompositions of $x$ as described above.

\begin{remark} 
 In contrast with the situation of $\h_{p,\at}^{c}(\M)$ and $\h_{p,\cat}^c(\M)$ for $1<p<2$, the above definition of $\h_{p,\alg}^c(\M)$ for $1<p<2$ works out well. The reason behind is the fact that any element of the unit ball of $\h_{p,\alg}^c(\M)$ is already an algebraic $\h_p^c$-atom (see Theorem~\ref{main}).
\end{remark} 

It is clear that  $(p,2)_c$-atoms  are  $(p,2)_c$-crude atoms. Also, $(p,2)_c$-crude atoms are algebraic $\h_p^c$-atoms. Thus for  $0<p \leq  1$, the following inclusions hold:
\[
\h_{p,\at}^c(\M) \subseteq \h_{p,\cat}^c(\M) \subseteq \h_{p,\alg}^c(\M).
\]

\begin{remark}\label{alg-crude}
Every algebraic $\h_p^c$-atom $x$ is a combination of at most countably many $(p,2)_c$-crude atoms satisfying certain convergence properties. More precisely, 
\[
x =\sum_{n} \lambda_n a_n\;\text{ (convergence in }\; L_p(\M)),
\] 
where $a_n$'s are $(p,2)_c$-crude atoms  or elements of the unit ball of $L_p(\M_1)$, and $\lambda_k \in \mathbb{C}$ satisfying 
\[
\sum_n|\lambda_n| \le1\;\text{ for }\;0<p \leq 1\quad\text{and}\quad
\sum_n|\lambda_n|^p \le1\;\text{ for }\;1<p <2.
\]
\end{remark} 

Indeed, let $x$ be an algebraic $\h_p^c$-atom. Then 
 $$x=\sum_{n\geq 1} y_nb_n$$
with $(y_n)$ and $(b_n)$ as in Definition~\ref{alg}. Let
 $$a_n=\frac{y_nb_n}{\|y_n\|_2\|b_n\|_q}\;\text{ and }\; \l_n=\|y_n\|_2\|b_n\|_q.$$
Then  $a_n$'s are $(p,2)_c$-crude atoms and
 $$x =\sum_{n} \lambda_n a_n.$$
 Moreover, for $0<p\le 1$ ($q\le2$), by Lemma~\ref{com-l2} we have
 \[
 \sum_n|\lambda_n| \le \Big(\sum_n\|y_n\|_2^2\Big)^{1/2}\,\Big(\sum_n\|b_n\|_q^2\Big)^{1/2}
 \le\Big\| \Big( \sum_{n\geq 1} |b_n|^2 \Big)^{1/2}\Big\|_q\le 1.
 \]
On the other hand, if $1<p<2$, then $q\ge 2$, so
  \[
  \Big(\sum_n|\lambda_n|^p \Big)^{1/p}\le \Big(\sum_n\|y_n\|_2^2\Big)^{1/2}\,\Big(\sum_n\|b_n\|_q^q\Big)^{1/q}
 \le\Big\| \Big( \sum_{n\geq 1} |b_n|^2 \Big)^{1/2}\Big\|_q\le 1.
 \]

\medskip

Let us now discuss the connections between  the three atomic Hardy spaces described above  and  the Hardy spaces $\h_p^c(\M)$  and $\H_p^c(\M)$ from the previous section. It is easy to verify that if $a$ is a $(p,2)_c$-crude atom then $\|a\|_{\h_p^c} \leq 1$  and $\|a\|_{\H_p^c} \leq 1$ (see \cite[Lemma~4.2]{Hong-Mei} for $p=1$). This property extends to algebraic $\h_p^c$-atoms.
The  inclusion  in the next lemma was proved in \cite[Section~3.6]{Perrin2} for the case $1\leq p <2$. The argument used there carries over to the full range.  Since this is very essential in our discussion, we reproduce  it here for the convenience of the reader.

\begin{lemma}\label{lemma-alg} For $0<p<2$, we have
$
\h_{p,\alg}^c(\M) \subseteq \h_p^c(\M)$ and $\h_{p,\alg}^c(\M) \subseteq \H_p^c(\M)$. More precisely, suppose $x$ is an operator that admits a decomposition $x=x_1 + \sum_{n=1}^\infty a_n b_n$ satisfying:
\begin{enumerate}[{\rm (i)}]
\item $x_1 \in L_p(\M_1)$;
\item  for every $n\geq 1$,  $a_n \in L_2(\M)$, $\E_n(a_n)=0$, and $b_n \in L_q(\M_n)$ where $1/p=1/2 + 1/q$.
\end{enumerate}
Then
\[
\max\Big\{\big\|x\big\|_{\H_p^c}, \big\|x\big\|_{\h_p^c}\Big\} \leq
\Big( \big\|x_1\big\|_p^p + \sum_{n\geq 1} \big\|a_n \big\|_2^2 \Big)^{1/2}  \Big\| \Big(|x_1|^{2-p} + \sum_{n\geq 1} \big|b_n\big|^2 \Big)^{1/2} \Big\|_q.
\]
In particular, if  $x$ is an algebraic $\h_p^c$-atom,  then 
 \[
\big\|x \big\|_{\h_p^c} \leq 1 \ \ \text{and}\ \  \big\|x \big\|_{\H_p^c} \leq 1.
\]
\end{lemma}
\begin{proof}
We provide the proof for the  $\h_p^c$-norm. The adjustment to  the $\H_p^c$-norm is straightforward.
Let $x= x_1+ \sum_{n\geq 1} a_nb_n$  with
$\E_n(a_n)=0$ and $b_n \in L_q(\M_n)$ for all $n\geq 1$.  We may assume by approximation that the $a_n$'s and $b_n$'s are bounded operators. Let $d_k=\E_k-\E_{k-1}$. We observe first that  $d_1(x)=x_1$ and for $k\geq 2$,
\[
d_k(x)=\sum_{n<k} d_k(a_n)b_n.
\]
It can be easily seen from the $a_n$'s  and $b_n$'s that
\begin{align*}
d_k(x) &=\sum_{n<k} \E_k(a_nb_n) -\sum_{n<k-1}\E_{k-1}(a_nb_n)\\
&= \sum_{n<k-1} d_k(a_n)b_n + \E_k(a_{k-1})b_{k-1}.
\end{align*}
 Since $\E_{k-1}(a_{k-1})=0$, the extra term is equal to $d_k(a_{k-1})b_{k-1}$. Thus, from the above form of $d_k(x)$, we deduce that   for $k\geq 2$,
\[
\E_{k-1}(|d_k(x)|^2) =\sum_{n,m <k}b_m^* \E_{k-1}\big( d_k(a_m)^* d_k(a_n)\big) b_n.
\]
The key part of the argument is  Junge's identification \cite{Ju} which states that for every $j\geq 1$,  there exists an isomorphic  right $\M_j$-module map $u_j: L_2(\M) \to L_2(\M_j \overline{\otimes}  B(\ell_2(\mathbb{N})))$ whose range is a closed subspace consisting of column vectors and satisfying the property that for  $y, z \in L_2(\M)$
\[
u_j(y)^*u_j(z)=\E_j(y^*z) \otimes e_{1,1}
\]
where $(e_{l,i})_{l,i\geq 1}$ denotes the unit matrices in $B(\ell_2(\mathbb{N}))$. It then follows that
\[
\E_{k-1}\big( d_k(a_m)^* d_k(a_n)\big) \otimes e_{1,1} =u_{k-1}(d_k(a_m))^* \cdot \ u_{k-1}(d_k(a_n)),
\]
where for  given $y$,  $u_{k-1}(y)$ is a column vector. Therefore,  
\[
\E_{k-1}(|d_k(x)|^2) \otimes e_{1,1}= \big| \sum_{n<k} u_{k-1} ( d_k(a_n)) \cdot (b_n \otimes e_{1,1}) \big|^2
\]
as operators affiliated with  $\M \overline{\otimes} B(\ell_2)$.

Let $\a_{1,1}:= |x_1|^{p/2}\otimes e_{1,1}$ and for $k\geq 2$, we set
   $\alpha_{k,n}  :=u_{k-1} ( d_k(a_n))$  when $1\leq n<k$.
Denote by  $A$  the lower triangular matrix $(\alpha_{k,n} )_{1\leq n<k}$ that takes  its  values in $L_2(\M \overline{\otimes} B(\ell_2))$.
   Multiplying $A$ from the right by the column
matrix $C= |x_1|^{1-p/2}\otimes e_{1,1}  \otimes e_{1,1}+ \sum_{n \geq 1} b_n \otimes  e_{1,1} \otimes e_{n+1,1}$, we get the column matrix:
\[
B=|x_1| \otimes e_{1,1}  \otimes e_{1,1}+\sum_{k\geq 2}\big[ \sum_{1\leq n<k} u_{k-1}(d_n(a_n))\cdot(b_n \otimes e_{1,1})\big] \otimes e_{k,1}.
\]
This shows  that
\begin{align*}
s_c^2(x) \otimes e_{1,1} \otimes e_{1,1}&= |x_1|^2 \otimes e_{1,1} \otimes e_{1,1} + \sum_{k\geq 2}  \big| \sum_{1\leq n<k} u_{k-1} ( d_k(a_n))\cdot (b_n \otimes e_{1,1}) \big|^2 \otimes e_{1,1}=B^*B\\
&=\Big| (\alpha_{k,n})_{1\leq n<k} \cdot \big( |x_1|^{1-p/2} \otimes e_{1,1} \otimes e_{1,1}+\sum_{j\geq 1}  b_j  \otimes e_{1,1} \otimes e_{j+1,1}\big)\Big|^2.
\end{align*}
That is, the conditioned square function of $x$ takes the following form:
 \[
s_c(x)\otimes e_{1,1} \otimes e_{1,1}=\big| (\alpha_{k,n})_{1\leq n<k}\cdot \big( |x_1|^{1-p/2} \otimes e_{1,1} \otimes e_{1,1}+\sum_{j\geq 1}  b_j  \otimes e_{1,1} \otimes e_{j+1,1}\big)\big|.
\]
 By H\"older's inequality, 
\begin{align*}
\big\|x\big\|_{\h_p^c} &\leq  \big\| (\alpha_{k,n})_{1\leq n<k}\big\|_{L_2(\M \overline{\otimes} B(\ell_2(\mathbb{N}^2))}  \big\| |x_1|^{1-p/2} \otimes e_{1,1} +\sum_{n\geq 1} b_n \otimes e_{n+1,1} \big\|_{L_q(\M\overline{\otimes} B(\ell_2(\mathbb{N})))}\\
&= \Big( \big\|\a_{1,1}\|_2^2 + \sum_{k\geq 2} \sum_{1\leq n<k} \big\| \alpha_{k,n}\big\|_{L_2(\M \overline{\otimes} B(\ell_2(\mathbb{N})))}^2 \Big)^{1/2}\,\Big\| \Big( |x_1|^{2-p}+\sum_{n\geq 1} |b_n|^2 \Big)^{1/2}\Big\|_q.
\end{align*}
Recall that for $k\geq 2$, $|\alpha_{k,n}|^2=\E_{k-1}(|d_k(a_n)|^2) \otimes e_{1,1}$. Taking into account that  $d_k(a_n)=0$ for $n\geq k$, we then deduce that
\begin{align*}
\Big( \big\|\a_{1,1}\big\|_2^2+\sum_{k\geq 2} \sum_{1\leq n<k} \big\| \alpha_{k,n}\big\|_2^2 \Big)^{1/2}
&=\Big(\big\|x_1\big\|_p^p +  \sum_{k\geq 2} \sum_{n\geq 1} \|d_k(a_n)\|_2^2 \Big)^{1/2}\\
&=\Big(\big\|x_1\big\|_p^p + \sum_{n\geq 1} \sum_{k\geq 2} \|d_k(a_n)\|_2^2 \Big)^{1/2}\\
&=\Big( \big\|x_1\big\|_p^p+\sum_{n\geq 1} \|a_n\|_2^2 \Big)^{1/2}.
\end{align*}
We have thus proved the desired estimate for $\|x\|_{\h_p^c}$.
\end{proof}

With the preceding  lemma, we can complete the series of  continuous inclusions which are valid for the full range $0<p<2$:
\begin{equation}\label{inclusion1}
\h_{p,\at}^c(\M) \subseteq \h_{p,\cat}^c(\M) \subseteq \h_{p,\alg}^c(\M) \subseteq \h_p^c(\M) \quad \text{for}\  0<p\leq 1
\end{equation}
and
\begin{equation}\label{inclusion2}
 \h_{p,\alg}^c(\M) \subseteq \h_p^c(\M) \quad \text{for}\  1<p <2.
\end{equation}

The general atomic decomposition problem for noncommutative martingales can be thought of as   determining  if these various  martingale Hardy spaces in the  respective inclusions  in   \eqref{inclusion1} and \eqref{inclusion2}  coincide. Our first result asserts that the reverse to the  first inclusion  in \eqref{inclusion1}  always holds. More precisely, we have:

\begin{proposition}\label{crude} 
Let $0<p <2$.
Then every $(p,2)_c$-crude atom $a$ can be decomposed into $(p,2)_c$-atoms: for any given $\b>1$ $a$ can be represented as
 \[
 a= \sum_{k=1}^\infty \lambda_k a_k,
 \]
 where the $a_k$'s are $(p,2)_c$-atoms and the $ \lambda_k$'s satisfy 
 \[
 \big(\sum_{k=1}^\infty\|\l_k|^p\big)^{1/p}\le\b.
 \]
 Consequently, $\h_{p,\at}^{c}(\M)= \h_{p,\cat}^{c}(\M)$ isometrically for $0<p\le1$.
\end{proposition}
\begin{proof}
 Let $a=yb$ with $\|y\|_2\leq 1$, $\E_n(y)=0$, and $b\in L_q(\M_n)$ with $\|b\|_q\leq 1$, where $1/q=1/p-1/2 $.
We may assume  (by considering polar decomposition) that $b\geq 0$. Fix $\b>1$ and
consider the sequence of mutually disjoint projections in $\M_n$ defined by
\[
 e_k=\chi_{[\b^k, \b^{k+1})}(b) \quad \text{for}\ k\in \Z.
\]
We write
 \begin{equation}\label{decomp-a}
 a= \sum_{k=-\infty}^\infty \lambda_k a_k,
 \end{equation}
where for  every $k \in \Z$, we define
\[
a_k= \frac{\T(e_k)^{-1/q}}{ \|ybe_k\|_{2}}ybe_k\  \quad  \text{and}\ \lambda_k =\|ybe_k\|_{2}\ . \  \T(e_k)^{1/q}.
\]
Thus,  each  $a_k$ is clearly a $(p,2)_c$-atom and 
the above series converges in $L_p(\M)$.
We claim that $\sum_{k\in \Z} |\lambda_k|^p \leq \b^p$.
First, by H\"older's inequality, we have
\[
\Big(\sum_{k\in \Z} |\lambda_k |^p\Big)^{1/p} \leq \Big( \sum_{k\in \Z} \b^{qk} \T(e_k) \Big)^{1/q}.
\Big( \sum_{k\in \Z} \b^{-2k} \|ybe_k\|_{2}^2\Big)^{1/2}.
\]
Next,  since $e_k$ commutes with $b$, we have the
 following simple estimate:
 \begin{align*}
 \big\|ybe_k\big\|_{2}^2
 &=\big\|y (e_kbe_k)\|_2^2\\
 &\leq \b^{2(k+1)} \big\| ye_k\big\|_2^2.
 \end{align*}
Thus,  we  deduce that
 \begin{align*}
 \Big(\sum_{k\in \Z} |\lambda_k|^p \Big)^{1/p} &\leq \Big( \sum_{k\in \Z} \b^{qk} \T(e_k) \Big)^{1/q}.
 \Big(\b^2\sum_{k\in \Z}\big\| ye_k\big\|_{2}^{2}\Big)^{1/2} \\
 &\leq \b\|b\|_q \big\|y\big\|_{2} \leq \b.
\end{align*}
For the case $0<p\leq 1$, the above assertion clearly implies  that $\h_{p,\at}^c(\M)$ and $\h_{p,\cat}^p(\M)$ are isometric. 
\end{proof}


The  next theorem is the main result  of this section. It  shows that the martingale  Hardy space $\h_p^c(\M)$ admits atomic decomposition when algebraic atoms are used.
It extends \cite[Theorem~3.6.13]{Perrin2} to the full range $0<p<2$.

\begin{theorem}\label{main} 
 Let $0<p<2$. Then  \[\h_{p,\alg}^c(\M)= \h_p^c(\M) \quad \text{with equivalent (quasi) norms.}\]
 More precisely,  if
 $x \in \h_p^{c}(\M)$,  then $x$  admits a unique decomposition  $x=x_1 + y$ where $x_1\in L_p(\M_1)$ and $y$ is a scalar  multiple of an algebraic $\h_p^c$-atom. Moreover, if $\lambda$ is the scalar such that $\lambda^{-1}y$ is an algebraic $\h_p^c$-atom, then 
 \begin{equation}\label{l1-estimate}
 \big\|x_1\big\|_p + |\lambda| \leq \sqrt{2/p}\,\big\|x \big\|_{\h_p^c}.
 \end{equation}
Consequently, we have
\[
\big\|x \big\|_{\h_p^c} \leq \big\|x \big\|_{\h_{p,\alg}^c} \leq\max\big\{1, 2^{(1-p)/p} \big\} \sqrt{2/p}\, \big\|x \big\|_{\h_p^c}.
\]
\end{theorem}

\begin{proof} First, we note  from the definition and Lemma~\ref{lemma-alg} that every algebraic $\h_p^c$-atom belongs to  $\h_p^{0,c}(\M)$. Since we have the direct sum $\h_p^c(\M)= L_p(\M_1) \oplus \h_p^{0,c}(\M)$,  it follows that if such decomposition exists, then  it is unique.

We only    need to prove the inclusion  $\h_{p}^{c}(\M) \subseteq \h_{p,\alg}^{c}(\M)$ as the reverse inclusion  is exactly  Lemma~\ref{lemma-alg}.
 It suffices to verify this for finite martingales in $\M$.
Fix a finite martingale $x=(x_n)_{n\geq 1}$ in $ \M$.  By approximation, we  assume that each of the $s_{c,n}(x)$'s ($n\geq 1$) is  invertible with bounded inverse.  We will denote $s_{c,n}(x)$ simply by $s_n$.  We now describe a concrete decomposition of $x$.  We begin by writing:
\[
x =  \sum_{n\geq 1} dx_ns_n^{-2 +p}s_n^{2-p}.
\]
Taking  $s_0=0$, we have
\begin{align*}
x&= \sum_{n\geq 1} dx_n s_n^{-2+p}\big[ \sum_{1\leq j\leq n} (s_j^{2-p}-s_{j-1}^{2-p})\big]\\
&=  \sum_{j\geq 1} \sum_{n\geq j} dx_n s_n^{-2 +p}(s_j^{2-p}-s_{j-1}^{2-p})\\
&=\sum_{n\geq 1} dx_ns_n^{-2+p} s_1^{2-p} +
\sum_{n\geq 2} dx_ns_n^{-2+p} (s_2^{2-p}- s_1^{2-p}) +
 \sum_{j\geq 3} \sum_{n\geq j} dx_n s_n^{-2 +p}(s_j^{2-p}-s_{j-1}^{2-p})\\
 &=x_1 + \big(\sum_{n\geq 2} dx_ns_n^{-2+p} \big)s_2^{2-p} + \sum_{j\geq 3} \sum_{n\geq j} dx_n s_n^{-2 +p}(s_j^{2-p}-s_{j-1}^{2-p})\\
 &=x_1 +y.
\end{align*}
Clearly, $x_1 \in L_p(\M_1)$ and we claim that $y$ is a scalar multiple of an algebraic $\h_p^c$-atom.  To verify this claim, we consider the following sequences of operators:

\begin{equation}\label{spliting}
\begin{cases}
\a_1&:=\displaystyle{\sum_{n\geq 2} dx_n s_n^{-2+p} s_2^{1-(p/2)}};\\
\a_l &:=\displaystyle{ \sum_{n\geq l+1} dx_n s_n^{-2+p}(s_{l+1}^{2-p}-s_{l}^{2-p})^{1/2} \quad  \text{for $l\geq 2$}};\\
\b_1 &:= s_2^{1-(p/2)};\\
\b_l &:=\displaystyle{(s_{l+1}^{2-p}-s_{l}^{2-p})^{1/2} \quad \text{for $l\geq 2$}}.
\end{cases}
\end{equation}
Then
\[
y= \sum_{l\geq 1} \a_l\b_l.
\]
 We begin by observing that
since $(s_n)_{n\geq 1}$ is a predictable sequence, we have for every $l\geq 1$,
$\E_l(\a_l)=0$. Also, for every $l\geq 1$,  $\b_l \in L_q(\M_l)$ where $1/p=1/2 +1/q$. Moreover, we have the following estimates on the $L_2$-norms of the  sequence $(\a_l)_{l\geq 1}$:
\begin{align*}
\sum_{l\geq 1} \big\|\a_l\big\|_2^2 &=\Big\|\sum_{n\geq 2} dx_n s_n^{-2+p} s_2^{1-(p/2)}\Big\|_2^2 + \sum_{l\geq 2} \Big\| \sum_{n\geq l+1} dx_n s_n^{-2+p}(s_{l+1}^{2-p} -s_l^{2-p})^{1/2} \Big\|_2^2\\
 &=\sum_{n\geq 2} \big\|dx_n s_n^{-2+p} s_2^{1-(p/2)}\big\|_2^2 + \sum_{l\geq 2}  \sum_{n\geq l+1} \big\|dx_n s_n^{-2+p}(s_{l+1}^{2-p} -s_l^{2-p})^{1/2} \big\|_2^2\\
&=\sum_{n\geq 2} \T\big( dx_n s_n^{-2+p} s_2^{2-p}s_n^{-2 +p} dx_n^*\big) +
\sum_{l\geq 2}  \sum_{n\geq l+1} \T\big(dx_n s_n^{-2+p} (s_{l+1}^{2-p}-s_l^{2-p})s_n^{-2 +p} dx_n^*\big).
\end{align*}
Interchanging the summations on the second quantity,
\begin{align*}
\sum_{l\geq 1} \big\|\a_l\big\|_2^2&=\sum_{n\geq 2} \T\big( dx_n s_n^{-2+p} s_2^{2-p}s_n^{-2 +p} dx_n^*\big) +
\sum_{n\geq 3}   \T\big(dx_n s_n^{-2+p}\big[ \sum_{l=2}^{n-1}(s_{l+1}^{2-p}-s_l^{2-p})\big]s_n^{-2 +p} dx_n^*\big)\\
&=\sum_{n\geq 2} \T\big( dx_n s_n^{-2+p} s_2^{2-p}s_n^{-2 +p} dx_n^*\big) +
\sum_{n\geq 3}   \T\big(dx_n s_n^{-2+p}(s_{n}^{2-p}-s_2^{2-p})s_n^{-2 +p} dx_n^*\big)\\
&=\T\big( dx_2 s_2^{-2+p} dx_2^*\big) +\sum_{n\geq 3}   \T\big(dx_n s_n^{-2+p} dx_n^*\big)\\
&=\sum_{n\geq 2} \T\big( s_n^{-2 +p}(s_n^2 -s_{n-1}^2) \big).
\end{align*}
According to Lemma~\ref{mart-2}, this leads to the estimate
\begin{equation}\label{norm-a}
 \sum_{l\geq 1} \big\|\a_l \big\|_2^2 \leq  \frac{2}{p} \sum_{n\geq 2} \T\big( s_n^p -s_{n-1}^p\big)=
 \frac{2}{p}  \Big(\big\|x\big\|_{\h_p^c}^p-\big\|x_1\big\|_p^p\Big).
\end{equation}
On the other hand, for the sequence $(\b_l)_l$, we have:
\begin{align*}
\Big\| \Big( \sum_{l\geq 1} |\b_l|^2 \Big)^{1/2} \Big\|_q &= \Big\| \Big(s_2^{2-p} + \sum_{l\geq 2} (s_{l+1}^{2-p} -s_l^{2-p}) \Big)^{1/2} \Big\|_q \\
&=\big\| s^{1 -(p/2)} \big\|_q\\
&=\big\|x \big\|_{\h_p^c}^{p/q}.
\end{align*}
Combining this last estimate with \eqref{norm-a}, we conclude that
\[
\Big(\sum_{l\geq 1} \big\|\a_l\big\|_2^2\Big)^{1/2} \Big\| \Big( \sum_{l\geq 1} |\b_l|^2 \Big)^{1/2} \Big\|_q  \leq  \sqrt{2/p}\, \Big(\big\|x\big\|_{\h_p^c}^p-\big\|x_1\big\|_p^p\Big)^{1/2}\big\|x\big\|_{\h_p^c}^{p/q}.
\]
This shows that if  we set $\lambda= \sqrt{2/p}\, \Big(\big\|x\big\|_{\h_p^c}^p-\big\|x_1\big\|_p^p\Big)^{1/2}\big\|x\big\|_{\h_p^c}^{p/q}$, then  by definition,   the operator $a=\lambda^{-1}y$ is an algebraic $\h_p^c$-atom and therefore we have the desired decomposition. It remains to verify the norm estimates.
We have:
\begin{align*}
\big\|x_1 \big\|_p  +|\lambda|&=  \big\|x_1\big\|_p + \sqrt{2/p}\, \Big(\big\|x\big\|_{\h_p^c}^p-\big\|x_1\big\|_p^p\Big)^{1/2}\big\|x\big\|_{\h_p^c}^{p/q}\\
&\leq  \Big[ \big\|x_1\big\|_p^{p/2} + \sqrt{2/p}\, \Big(\big\|x\big\|_{\h_p^c}^p-\big\|x_1\big\|_p^p\Big)^{1/2}\Big]\big\|x\big\|_{\h_p^c}^{p/q}.
\end{align*}
 We will verify that  the  last estimate  is  further majorized by $\sqrt{2/p}\, \big\|x \big\|_{\h_p^c}$. Indeed, let $b=\|x\|_{\h_p^c}^{p/2}$ and consider the function  defined  by
 \[
 f(t)=t + \sqrt{2/p}\,\big(b^2- t^2\big)^{1/2} \quad \text{for}\  t \in [0,b].
 \]
 One can check that $f$ attains its maximum at $t=0$.  That is,  for  every $ t\in [0,b]$,
 \[
 f(t) \leq \sqrt{2/p}\,b= \sqrt{2/p}\, \big\|x\big\|_{\h_p^c}^{p/2}.
 \]
  We can now conclude that
 \[
\big\|x_1\big\|_{p}  +|\lambda| \leq f\big(\big\|x_1\big\|_p^{p/2}\big) \big\|x\big\|_{\h_p^c}^{p/q}  \leq \sqrt{2/p}\,\big\|x \big\|_{\h_p^c},
 \]
 which is inequality \eqref{l1-estimate}. For $1\le p<2$, \eqref{l1-estimate}  already gives $\|x\|_{\h_p^c} \leq \sqrt{2/p}\,\big\|x \big\|_{\h_p^c}$. 
 For  the case  $0<p<1$, we have 
  \begin{align*}
 \big\|x \big\|_{\h_{p,\alg}^c} 
 \leq \big(\big\|x_1\big\|_p^p + |\lambda|^p\big)^{1/p}
 \leq  2^{(1-p)/p}\big(\big\|x_1\big\|_p + |\lambda| \big) 
\leq  2^{(1-p)/p} \sqrt{2/p}\big\|x\big\|_{\h_p^c}. 
 \end{align*}
 This concludes the proof. 
\end{proof}

It was shown in \cite[Lemma~3.6.8]{Perrin2} (see also  \cite[Lemma~3.1]{Hong-Junge-Parcet}) that for $1\leq p<2$, the set of algebraic $\h_p^c$-atoms are already absolutely convex.  Theorem~\ref{main} captures this phenomenon for the full range. 
More precisely, it implies that if  $0<p<1$,  $(a_k)_{k\geq 1}$ is a sequence of  algebraic $\h_p^c$-atoms, and $(\lambda_k)_{\geq 1}$ is a sequence of scalars satisfying $\sum_{k\geq 1} |\lambda_k|^p <\infty$,   then $x=\sum_{k\geq 1} \lambda_k a_k$ 
is a scalar multiple of algebraic $\h_p^c$-atom.


\medskip


At the time of this writing, it is still open  if the algebraic atomic Hardy space $\h_{p,\alg}^c(\M)$  coincides with the atomic Hardy space  $\h_{p,\at}^c(\M)$ (equivalently, $\h_{p,\cat}^c(\M)$) for $0<p<1$. However, combined with Proposition~\ref{crude} and Remark~\ref{alg-crude}, the previous theorem implies the following weaker form of atomic decomposition which is sufficient for some applications.

\begin{corollary}\label{weak-decomp} Let $0< p <2$.  Every $x \in \h_p^c(\M)$ admits a  decomposition $\displaystyle{x=x_1 +\sum_{l\geq 1} \lambda_l a_l}$ where
\begin{enumerate}[{\rm (i)}]
\item $x_1 \in L_p(\M_1)$;
\item for each $l\geq 1$, $a_l$ is a $(p,2)_c$-atom and $\lambda_l \in \mathbb{C}$;
\item the series $\sum_{l\geq 1} \lambda_l a_l$ converges in  $\h_p^c(\M)$;
\item the following inequality holds:
\[
\big\|x_1\big\|_p^{\max\{1,p\}} + \sum_{l\geq 1} |\lambda_l |^{\max\{1,p\}} \leq \big(\sqrt{2/p}\, \big\|x \big\|_{\h_p^c}\big)^{\max\{1,p\}}.
\]
\end{enumerate}
\end{corollary}

\bigskip

For the case $1\leq p<2$, the preceding corollary  when coupled with Proposition~\ref{crude} provides  constructive
 proofs of all  atomic decompositions from \cite{Bekjan-Chen-Perrin-Y}.
  In particular,  it solves \cite[Problem~1]{Bekjan-Chen-Perrin-Y}. We state this explicitly in the next result.
\begin{corollary}\label{atom-1}
We have
\[
\h_1^c(\M)=\h_{1,\at}^c(\M) \quad \text{with equivalent norms}.
\]
More precisely,  if $x \in \h_1^c(\M)$, then
\[
\big\| x \big\|_{\h_1^c} \leq  \big\|x \big\|_{\h_{1,\at}^c} \leq  \sqrt{2}\, \big\|x \big\|_{\h_1^c}.
\]
The constant $\sqrt{2}$ is optimal.
\end{corollary}
Details  of the construction are left to the reader. We only point out  the fact that
 the constant $\sqrt{2}$ is the best possible which  follows from the trivial  inequality  $\|x\|_1 \leq \|x\|_{\h_{1,\at}^c}$ and \cite[Theorem~4.11]{Jiao-Ran-Wu-Zhou}
(see also Corollary~\ref{Best} below).

\begin{remark}
Using the constructive approach to the noncommutative Davis decomposition \cite{Junge-Perrin, Ran-Wu-Xu} and our proof of atomic decomposition for $\h_1^c(\M)$ in Corollary~\ref{atom-1},  one can explicitly express the decomposition of any element of $\H_1^c(\M)$ into \emph{$2$-atomic blocks} in the sense of  Conde-Alonso and Parcet (see \cite[Theorem~1.1]{Alonso-Parcet}).
\end{remark}

\medskip

We  take the opportunity to present below a simple  approach to
   sharp inequalities  between $L_p$-norms  and $\h_p^c$-norms when $0<p<2$ based on the construction used in the proof of Theorem~\ref{main}. The next result was obtained  recently in \cite{Jiao-Ran-Wu-Zhou} and  is a
 noncommutative analogue of a sharp inequality from \cite{Wang}.  We refer to  the monograph \cite{OS}  for extensive discussions on the importance of  sharp inequalities in  classical  martingale theory.
\begin{corollary}[{\cite[Theorem~4.11]{Jiao-Ran-Wu-Zhou}}]\label{Best}
Let  $0<p < 2$. For every $x \in \h_p^c(\M)$, the following inequality holds:
\[
\big\|x \big\|_p \leq  \sqrt{2/p}\, \big\| x \big\|_{\h_p^c}.
\]
The constant $\sqrt{2/p}$ is the best possible.
\end{corollary}
\begin{proof}  Let $x  \in \M$. By approximation, we assume that  for every $n\geq 1$, $s_{c,n}(x)$  is invertible with bounded inverse. As above, we denote $s_{c,n}(x)$ by $s_n$ and we take $s_0=0$. We write $x=\sum_{l\geq 1} a_lb_l$ with $a_l=\sum_{n\geq l} dx_n s_n^{-2+p}(s_l^{2-p}-s_{l-1}^{2-p})^{1/2}$ and $b_l=(s_l^{2-p}-s_{l-1}^{2-p})^{1/2}$. As slight difference here is that we do not need to isolate  the first term $x_1$ since we do not  require any particular properties on the sequences $(a_l)_{l\geq 1}$ and $(b_l)_{l\geq 1}$ beside their respective norms. Using H\"older's inequality, we may deduce that:
\begin{align*}
\big\|x\big\|_p &= \Big\| \sum_{l\geq 1} a_l b_l\Big\|_p\\
&\leq \Big\| \Big( \sum_{l\geq 1} a_l a_l^*\Big)^{1/2} \Big\|_2 . \Big\|\Big( \sum_{l\geq 1} b_l^* b_l\Big)^{1/2}\Big\|_{2p/(2-p)}\\
&=\Big( \sum_{l\geq 1}\big\|a_l\big\|_2^2\Big)^{1/2} \big\| s^{1-(p/2)}\big\|_{2p/(2-p)}\\
&=\Big( \sum_{l\geq 1}\big\|a_l\big\|_2^2\Big)^{1/2}\big\|x\big\|_{\h_p^c}^{1-(p/2)}.
\end{align*}
Proceeding as in  the proof of Theorem~\ref{main},  we have
\[
 \sum_{l\geq 1}\big\|a_l\big\|_2^2 =\sum_{n\geq 1} \T\big(s_n^{-2+p}(s_n^2-s_{n-1}^2)\big) \leq \frac{2}{p} \big\|x\big\|_{\h_p^c}^p.
\]
This clearly yields  the desired inequality.
The fact that the above constant is sharp is already the case for classical martingales
as shown in \cite{OS,Wang}.
\end{proof}

\begin{remark}
In \cite[Theorem~4.11]{Jiao-Ran-Wu-Zhou}, it was also proved that for $0<p<2$, the following sharp inequality holds:
\[
\big\|x \big\|_{\H_p^c} \leq  \sqrt{2/p}\, \big\| x \big\|_{\h_p^c}.
\]
We were  able to verify this through the decomposition used above only when $x_1=0$. That is,  for every $x\in \h_p^{0,c}(\M)$. The general case does not appear to follow from our construction.
\end{remark}

\begin{remark}
 All results in  this section easily extend to semifinite von Neumann algebras with minor modifications. Moreover, some of them remain valid in the type III
case. We refer to \cite{JX} for the definitions of noncommutative martingales and Hardy spaces in a $\s$-finite von Neumann algebra $\M$. The $(p,2)_c$-crude atoms  and algebraic $\h_p^c$-atoms are defined exactly in the same way. Using Haagerup's reduction theorem \cite{Haagerup-Junge-Xu}, we can show that the corresponding atomic Hardy space $\h_{p,\alg}^{c}(\M)$ coincides with $\h_p^c(\M)$ for all $0<p<2$.
 \end{remark}


\section{$(p,\infty)$-atoms}\label{(p,infty)-atom}


We begin with the definition of  $(p,q)$-atoms that extends the concept of $(p,2)$-atoms considered in the previous section.

\begin{definition}\label{df:pq-atom}
Let $0 < p < 2$ and  $\max(p, 1)< q \le \8.$ An operator $a \in L_p (\M)$ is called a $(p, q)_c$-atom, if there exist $n \ge 1$ and a projection $e \in \M_n$ such that:
\begin{enumerate}[{\rm (i)}]

\item $\E_n (a) =0;$

\item $r(a) \le e;$

\item $\| a \|_{\h^c_q} \le \tau (e)^{1/q-1/p}.$

\end{enumerate}
\end{definition}

The concept of $(p, q)_c$-atoms was introduced in \cite{Hong-Mei} (for $p=1$). However, the notion of $(p, \8)_c$-atoms is new and exactly the noncommutative analogue of the so-called {\it simple} atom in the classical setting (see \cite[Definition 2.4]{Weisz}). Note that, the associated $(p, \8)_c$-atom in \cite{Hong-Mei} was defined by using $\| a \|_{\mathrm{bmo}^c}$ in place of $\| a \|_{\h^c_{\8}}$ in $(\mathrm{iii}),$ which we may call instead a $(p, \mathrm{bmo})_c$-atom for the sake of convenience. Clearly, $(p, q_1)_c$-atoms are necessarily $(p, q_2)_c$-atoms whenever $0 <p <2$ and $\max(p, 1)< q_2 < q_1 \le \8.$ On the other hand, a $(p, \8)_c$-atom is a $(p, \mathrm{bmo})_c$-atom.

\begin{definition}\label{df:pq-atomicHardySpace}
Let $0 < p \le 1< q \le \8.$  
Let $\h_{p,\,\at_q}^c(\M)$  be the space of all $x \in L_p (\M)$ which admits a decomposition
\be
x = \sum_k \lambda_k a_k \ \text{(convergence in $L_p(\M)$)},
\ee
where for each $k,$ $a_k$ is a $(p, q)_c$-atom or an element in the unit ball of $L_p (\M_1),$ and $\lambda_k \in \mathbb{C}$ satisfying $\sum_k |\lambda_k|^p < \8$.  $\h_{p,\,\at_q}^c(\M)$ is equipped with the $p$-norm:
\be
\| x \|_{\h^c_{p,\,\at_q}} = \inf \Big ( \sum_k |\lambda_k|^p \Big )^{1/p}\,,
\ee
where the infimum is taken over all decompositions of $x$ described above. 
\end{definition}

By definition, $\h^c_{p, \,\at_2} (\M) = \h^c_{p, \mathrm{at}} (\M)$ for all $0<p \le1.$ As in the case of   $\h^r_{p, \mathrm{at}} (\M)$ and $\h_{p, \mathrm{at}} (\M)$ defined in the previous section, we may also define the row version $\h^r_{p, \,\at_q} (\M)$ and the mixed version $\h_{p, \,\at_q} (\M).$  We omit the details.

One can check that $\h^c_{p, \at_q} (\M)\subset \h^c_p (\M)$ for any $0<p\le 1<q\le\8$. On the other hand, it follows from Definition \ref{df:pq-atomicHardySpace} that for $0 < p \le 1$ and $2< q < \8$  
\beq\label{inclusion-infty}
\h^c_{p, \,\at_\8} (\M) \subseteq \h^c_{p, \,\at_q} (\M) \subseteq \h^c_{p, \mathrm{at}} (\M).
\eeq
 The following theorem shows that the reverse inclusion holds too, so $\h^c_{p, \,\at_q} (\M) = \h^c_{p, \mathrm{at}} (\M)$ for  $0 < p \le 1$ and $2< q < \8$. The latter equality was proved in \cite{Hong-Mei} for $p=1$ and $1<q<\8$. 

\smallskip

The proof of the decomposition in the following theorem also works for $1<p<2$, so we state it for the full range $0<p<2$. However, at the time of this writing, we cannot prove the same result for the case $1<q<2$. 

\begin{theorem}\label{th:infty-atomicDecomp}
Let $0< p < 2.$ Then every $(p,2)_c$-atom $a$ admits a decomposition: 
\be
a= \sum_k \lambda_k a_k\quad (\text{converges in } L_p (\M)),
\ee
where  each $a_j$ is a $(p,\infty)_c$-atom, and $\lambda_j \in \mathbb{C}$ such that
\be
\sum_j |\lambda_j |^p \le3^{p/2}\,\frac{\l^2-2\l^p}{\l^{2-p}-4}\,,
\ee
where $\l$ is any constant satisfying $\l^{2-p}>4$.
Consequently, $\h^c_{p,\mathrm{at}} (\M) = \h^c_{p, \,\at_\8} (\M)$ for all $0<p\le1$.
\end{theorem}

\begin{proof}
 This proof is quite elaborate. We divide it into five steps. The idea of the double truncation by Cuculescu's projections in Steps 1 and 2 below comes from  \cite{PR2006} on the noncommutative Gundy decomposition. During the whole proof, $a$ will be a fixed $(p,2)_c$-atom with the associated projection  $e \in \M_n$ such that i) $\E_n (a) =0,$ ii) $r (a) \le e,$ and iii) $\| a \|_2 \le \tau (e)^{1/2-1/p}$. Given $\e>0$ choose an increasing sequence $(N_k)_{k\ge 1}$ of integers with $N_1>n$ such that 
 $$\Big(\sum_{k=1}^\8\big\|\E_{N_{k+1}}(a)-\E_{N_{k}}(a)\big\|_2^p\Big)^{1/p}<\e$$
 for $0<p\le 1$ and 
 $$\sum_{k=1}^\8\big\|\E_{N_{k+1}}(a)-\E_{N_{k}}(a)\big\|_2<\e$$
for $1<p<2$. Then  $\E_{N_{1}}(a)$ and $\E_{N_{k+1}}(a)-\E_{N_{k}}(a)$ for all $k\ge1$ satisfy the same properties as $a$. If the assertion holds for these operators,  it does so for $a$. Thus in the sequel we will  additionally assume that $a\in\M_N$ for some $N>n$; then the associated  martingale $(\E_k(a))_{k}$ is finite.

 \medskip\noindent{\it Step $1$}. We put $b=\tau (e)^{1/p}a.$ Then
\begin{enumerate}[$\bullet$]

\item $\E_n (b)=0$;

\item $r (b) \le e$;

\item $\| b \|^2_2 \le \tau (e).$
\end{enumerate}
Note that $s_{c, k}(b)=0$ for $k\le n$.  Fix $\lambda$ such that $\l^{2-p}>4$. We apply the construction of Cuculescu's projections to the supermartingale $(s_{c, k}^2 (b))_{k \ge n}$  and the parameter $\lambda^2$ to obtain a decreasing sequence  $(q_k)_{k\ge n}$ of projections  in $\M$ satisfying the following properties:
\begin{enumerate}[$\bullet$]

\item $q_n = e$  and $q_k \le e$ for all $k>n$;

\item $q_k \in \M_{k-1}$ for every $k>n$;

\item  $q_k$ commutes with $q_{k-1} s_{c, k}^2 (b) q_{k-1}$ for all $k>n$;

\item $q_k s_{c, k}^2 (b) q_k \le \lambda^2 q_k$ for all $k>n$; 

\item if we set $q = \bigwedge_{k \ge n} q_k,$ then $q\le e$ and
\be
\tau (e- q) \le \frac{1}{\lambda^2} \| s^2_c (b) \|_1 = \frac{1}{\lambda^2} \|b \|_2^2.
\ee
 \end{enumerate}
 It is worth to note that since $b\in\M_N$, $q_k=q_N$ for all $k\ge N$.  This remark applies to all similar constructions below.
 
 We consider the following martingale difference sequence:
\be
d y_k =0 \;\text{ for } 1\le k\le n\;\text{ and }\; d y_k = d b_k q_k\;\text{ for } k>n.
\ee
The corresponding  finite martingale $y = (y_k)_{k \ge 1}$ has the following properties:
\begin{enumerate}[$\bullet$]

\item $\E_n (y)=0$;

\item $r(y) \le e;$

\item $\| y \|^2_2 \le \| b \|^2_2;$ 
\item $\| \E_{k-1} [ |d y_k|^2 ] \|_{\infty} \le \lambda^2$ for all $k \ge 1.$

\end{enumerate}
The first three assertions are clear. The last is checked as follows:
\be\begin{split}
\E_{k-1} [ |d y_k|^2 ] = \E_{k-1} [ q_k | d b_k |^2 q_k ] \le q_k s^2_{c, k} (b) q_k
\le \lambda^2 q_k\,,\; k>n.
\end{split}\ee

Again, let $(\pi_k)_{k\ge n}$ be the sequence of Cuculescu's projections relative to the supermartingale $(s_{c, k}^2 (y))_{k \ge n},$ $\pi_n =e,$ and the parameter $\lambda^2$. Then $\pi_k \le e$ for all $k >n$; moreover, if we set $\pi = \bigwedge_{k \ge n} \pi_k,$ then $\pi \le e$ and
\be
\tau (e- \pi ) \le \frac{1}{\lambda^2} \| s^2_c (y) \|_1 \le \frac{1}{\lambda^2} \|b \|_2^2.
\ee
Letting
\be
g = \sum_{k > n} d y_k \pi_{k - 1},
\ee
we have that $g$ is a finite martingale such that
\begin{enumerate}[$\bullet$]
\item $\E_n (g)=0;$

\item $r (g) \le e;$ 

\item $\| s_c (g) \|_{\infty} \le \sqrt{3}\, \lambda.$

\end{enumerate}
Therefore, if we set
\be
a_{(0)} = \frac{1}{\sqrt{3}\, \lambda\, \tau (e)^{{1}/{p}}}\, g,
\ee
then $a_{(0)}$ is a $(p, \infty)_c$-atom with the associated projection $e \in \M_n.$

We need only to show the third assertion above. First, notice that for $k\ge n+1$
\be\begin{split}
s_{c,k}^2 (g)& = \sum^k_{j = n+1} \E_{j-1} [\pi_{j-1} |d y_j|^2 \pi_{j-1}]\\
& = \sum^k_{j = n+1} \big ( \pi_{j-1} s^2_{c,j} (y) \pi_{j-1} - \pi_{j-1} s^2_{c,j-1} (y) \pi_{j-1} \big )\\
& = \pi_k s^2_{c, k} (y) \pi_k + \sum^k_{j = n+1} \big ( \pi_{j-1} s^2_{c,j} (y) \pi_{j-1} - \pi_j s^2_{c,j} (y) \pi_j \big )\\
& = \pi_k s^2_{c, k} (y) \pi_k + \sum^k_{j = n+1} ( \pi_{j-1} - \pi_j ) s^2_{c,j} (y) ( \pi_{j-1} - \pi_j \big )\,,
\end{split}\ee
where the last equality follows from the commutativity between $\pi_j$ and $\pi_{j-1} s^2_{c,j} (y) \pi_{j-1}.$ Since $\| \E_{k-1} [|d y_k|^2] \|_{\infty} \le \lambda^2$, we have
\be\begin{split}
 \big\| s_{c, k}^2 (g) \big\|_{\infty} 
&\le \big\| \pi_k s^2_{c, k} (y) \pi_k \big \|_{\infty} + \Big \| \sum^k_{j = n} ( \pi_{j-1} - \pi_j ) \E_{j-1} [|d y_j |^2] ( \pi_{j-1} - \pi_j \big )\Big \|_{\infty}\\
& \quad + \Big \| \sum^k_{j = n} ( \pi_{j-1} - \pi_j ) s^2_{c,j-1} (y) ( \pi_{j-1} - \pi_j \big ) \Big \|_{\infty}\\
&\le \lambda^2 + \sup_{n \le j \le k} \big\| \E_{j-1} [|d y_j|^2]  \big\|_{\infty} + \lambda^2 \le 3 \lambda^2\,.
\end{split}\ee
Thus $ \| s_c (g) \|_{\infty} \le \sqrt{3}\, \lambda.$

\medskip\noindent{\it Step $2$}.   Let (with $\pi_{n-1}=e$)
\be
e^{(1)} = e - q \wedge \pi \quad \text{and} \quad e^{(1)}_{(i)} = q_i \wedge \pi_{i-1} - q_{i+1} \wedge \pi_i,\quad i \ge n.
\ee
Then $(e^{(1)}_{(i)})_{i \ge n}$ is a finite sequence of pairwise disjoint projections in $\M$ such that $e^{(1)}_{(i)}=0$ for $i>N$ and
\be
e^{(1)} = \sum_{i \ge n} e^{(1)}_{(i)}\,,\quad  e - q_k\wedge \pi_{k-1}= \sum^{k-1}_{i = n} e^{(1)}_{(i)}\;\text{ for } k >n.
\ee
Since $(e - q_k \pi_{k-1})( e - q_k\wedge \pi_{k-1})= e - q_k \pi_{k-1}$ for $k >n$, we have
\be\begin{split}
d b_k = d g_k + d b_k ( e - q_k \pi_{k-1} ) = d g_k + \sum^{k-1}_{i=n} d b_k ( e - q_k \pi_{k-1} ) e^{(1)}_{(i)}.
\end{split}\ee
Consequently, $b$ can be decomposed as
\beq\label{ortho1}
b = g + b^{(1)} = g + \sum_{i \ge n} b^{(1)}_{(i)}\,,
\eeq
where
\be
b^{(1)} = \sum_{k >n} d b_k ( e - q_k \pi_{k-1} ) 
\ee
and
\be
b^{(1)}_{(i)} = \sum_{k>i} d b_{k} (e - q_{k} \pi_{k -1}) e^{(1)}_{(i)}\,, \quad i \ge n.
\ee
 For $b^{(1)},$ we have the following properties:
\begin{enumerate}[$\bullet$]

\item $\E_n [b^{(1)}] =0$;

\item $\| b^{(1)} \|_2 \le 2 \| b \|_2$;

\item $r (b^{(1)}) \le e^{(1)}$ and
\be
\tau ( e^{(1)} ) \le \frac{2}{\lambda^2} \| b \|^2_2.
\ee
 \end{enumerate}
The last inequality follows from
\be
\tau (e^{(1)}) = \tau (e - q \wedge \pi ) \le \tau (e -q) + \tau (e- \pi) \le \frac{2}{\lambda^2} \| b \|^2_2\,.
\ee
On the other hand, every $b^{(1)}_{(i)}$ with $i\ge n$ satisfies the following properties:
\begin{enumerate}[$\bullet$]

\item $\E_i [b^{(1)}_{(i)}] =0$;

\item $e^{(1)}_{(i)} \in \M_i$ and $r (b^{(1)}_{(i)}) \le e^{(1)}_{(i)}$;

\item $ e^{(1)} = \sum_{i \ge n} e^{(1)}_{(i)}$ and
\be
\sum_{i \ge n} \tau ( e^{(1)}_{(i)}) = \tau (e^{(1)}) \le \frac{2}{\lambda^2} \| b \|^2_2\,;
\ee

\item $b^{(1)} = \sum_{i \ge n} b^{(1)}_{(i)}$;

\item $\sum_{i \ge n} \big\| b^{(1)}_{(i)}  \big\|^2_2= \big\| b^{(1)}  \big\|^2_2  \le 2 \| b \|^2_2$;

\end{enumerate}
All the assertions but the last one are clear from the construction. However, the proof of the last assertion follows immediately from the following two facts that $r (b^{(1)}_{(i)}) \le e^{(1)}_{(i)}$ for any $i \ge n,$ and $(e^{(1)}_{(i)})_{i \ge n}$ is a sequence of pairwise disjoint projections. 
 
 For the later reference, it is useful to note that the last sum  in \eqref{ortho1}  is  finite.

\medskip\noindent{\it Step $3$}. Up to now, we have completed the two first steps. In what follows, we will process these steps repeatedly. As shown above, for each $i \ge n,$ $b^{(1)}_{(i)}$ has the same properties satisfied by $b$ but with $e^{(1)}_{(i)}$ replacing $e.$  For each $i_1 \ge n,$ by repeating the above two steps applied to $(b^{(1)}_{(i_1)}\,,\,e^{(1)}_{(i_1)})$ instead of $(b,\,e)$  and the parameter $\l^4$, we find  $g_{(i_1)}, \;b^{(2)}_{(i_1)}, \;b^{(2)}_{(i_1 i_2)}, \;e^{(2)}_{(i_1)},$ and $e^{(2)}_{(i_1 i_2)} \in \M_{i_2}$ (with $i_2 \ge i_1$) such that
\be
 b^{(1)}_{(i_1)} = g_{(i_1)} +  b^{(2)}_{(i_1)} = g_{(i_1)} + \sum_{i_2 \ge i_1} b^{(2)}_{(i_1 i_2)}\,.
\ee
Using the arguments in steps 1 and 2, we see that all these operators satisfy the following properties:
\begin{enumerate}[{\rm 1)}]

\item $\E_{i_1} [g_{(i_1)}] = 0, r (g_{(i_1)}) \le e^{(1)}_{(i_1)},$ and $\|s_c (g_{(i_1)}) \|_{\infty} \le \sqrt{3}\, \l^2.$ Therefore, if
\be
a_{(i_1)} = \frac{1}{\lambda_{i_1}} \,g_{(i_1)}\quad \text{with} \quad \lambda_{i_1} = \sqrt{3} \,\l^2\, \tau (e^{(1)}_{(i_1)})^{1/p},
\ee
then $a_{(i_1)}$ is a $(p, \infty)_c$-atom with the associated projection $e^{(1)}_{(i_1)}.$

\item $(e^{(2)}_{(i_1 i_2)})_{i_2 \ge i_1 \ge n}$ is a finite family of pairwise disjoint projections such that
\be
e^{(2)}_{(i_1)} = \sum_{i_2 \ge i_1} e^{(2)}_{(i_1 i_2)}, \quad e^{(2)}_{(i_1)} \le e^{(1)}_{(i_1)},
\ee
and
\be
\sum_{i_2 \ge i_1} \tau ( e^{(2)}_{(i_1 i_2)}) = \tau (e^{(2)}_{(i_1)}) \le \frac{2}{\l^4}\, \big\| b^{(1)}_{(i_1)}\big\|^2_2\,.
\ee

\item $b^{(2)}_{(i_1)} =  \sum_{i_2 \ge i_1} b^{(2)}_{(i_1 i_2)}, \; r (b^{(2)}_{(i_1)}) \le e^{(2)}_{(i_1)}, \; \big\| b^{(2)}_{(i_1)} \big\|_2 \le 2 \big\|b^{(1)}_{(i_1)} \big\|_2$,\;
\be
 \big\| b^{(2)}_{(i_1)}  \big\|^2_2 = \sum_{i_2 \ge i_1}  \big\| b^{(2)}_{(i_1 i_2)}  \big\|^2_2\,.
\ee

\item $\E_{i_2} [b^{(2)}_{(i_1 i_2)}] =0$, and $r (b^{(2)}_{(i_1 i_2)}) \le e^{(2)}_{(i_1 i_2)}.$

\item The following decomposition holds:
\be
b = g + \sum_{i_1 \ge n} g_{(i_1)} + \sum_{i_2 \ge i_1 \ge n} b^{(2)}_{(i_1 i_2)}.
\ee
\end{enumerate}
Like \eqref{ortho1}, the two last sums above  are finite.

\medskip

Continuing this process inductively for $k \ge2$ with parameter $\l^{k+1}$,  we find  a  decomposition of $b$ into finite sums in $L_2(\M)$:
\beq\label{ortho5}
b = g + \sum_{i_1 \ge n} g_{(i_1)} + \cdots + \sum_{i_k \ge \cdots \ge i_1 \ge n} g_{(i_1 \cdots i_k)} + \sum_{i_{k+1} \ge i_k \ge \cdots \ge i_1 \ge n} b^{(k+1)}_{(i_1 \cdots i_k i_{k+1})}\,,
\eeq
as well as a finite family $\big\{e^{(k)}_{(i_1 \cdots i_k)}\in \M_{i_k}: i_k \ge \cdots \ge i_1 \ge n\big\}$ of pairwise disjoint projections in $\M$. Moreover, we have the following properties:
 \begin{enumerate}[{\rm i)}]

\item $\E_{i_k} [g_{(i_1 \cdots i_k)} ] =0,\; r (g_{(i_1 \cdots i_k)}) \le e^{(k)}_{(i_1 \cdots i_k)},$ and $\| s_c (g_{(i_1 \cdots i_k)}) \|_{\infty} \le \sqrt{3} \,\l^{k+1}.$ Hence, if
\beq\label{ortho5b}
a_{(i_1 \cdots i_k)} = \frac{1}{\lambda_{i_1 \cdots i_k}} \,g_{(i_1 \cdots i_k)}\quad \text{with} \quad \lambda_{(i_1 \cdots i_k)} = \sqrt{3} \, \l^{k+1} \tau (e^{(k)}_{(i_1 \cdots i_k)})^{1/p},
\eeq
then $a_{(i_1 \cdots i_k)}$ is a $(p, \infty)_c$-atom with the associated projection $e^{(k)}_{(i_1 \cdots i_k)}.$

\item $e^{(k+1)}_{(i_1 \cdots i_k)} = \sum_{i_{k+1} \ge i_k} e^{(k+1)}_{(i_1 \cdots i_k i_{k+1})},\; e^{(k+1)}_{(i_1 \cdots i_k)} \le e^{(k)}_{(i_1 \cdots i_k)}$ and
\beq\label{ortho6}
\sum_{i_{k+1} \ge i_k} \tau ( e^{(k+1)}_{(i_1 \cdots i_k i_{k+1})} ) = \tau ( e^{(k+1)}_{(i_1 \cdots i_k)} ) \le \frac{2}{\l^{2(k+1)}}\,   \
\big\| b^{(k)}_{(i_1\cdots i_k )}  \big\|^2_2.
\eeq

\item $b^{(k+1)}_{(i_1 \cdots i_k)} =  \sum_{i_{k+1} \ge i_k} b^{(k+1)}_{(i_1 \cdots i_k i_{k+1})},\; r (b^{(k+1)}_{(i_1\cdots i_k)}) \le e^{(k+1)}_{(i_1 \cdots i_k)},\; 
 \big\| b^{(k+1)}_{(i_1 \cdots i_k)}  \big\|_2 \le  2 \big\| b^{(k)}_{(i_1 \cdots i_k)}  \big\|_2 ,$
\beq\label{ortho6b}
  \sum_{i_{k+1} \ge i_k}  \big\| b^{(k+1)}_{(i_1 \cdots i_k i_{k+1})}  \big\|^2_2=\big\| b^{(k+1)}_{(i_1\cdots i_k)}  \big\|^2_2 
  \eeq
 and
  \beq\label{ortho6c} 
   \sum_{i_k \ge \cdots \ge i_1 \ge n} \big \| b^{(k)}_{(i_1 \cdots i_k)} \big \|^2_2 \le 4^{k} \big\| b \big\|^2_2.
\eeq

\item $\E_{i_{k+1}} [b^{(k+1)}_{(i_1 \cdots i_k i_{k+1})}] =0$ and $r (b^{(k+1)}_{(i_1 \cdots i_k i_{k+1})}) \le e^{(k+1)}_{(i_1 \cdots i_k i_{k+1})}.$
  \end{enumerate}
 
\noindent{\it Step $4$}.  In this step we show that the last sum in \eqref{ortho5} converges to zero in $L_p(\M)$ as $k\to\8$.   To that end, we first claim that
  \be\begin{split}
  \Big\| \sum_{i_{k+1} \ge i_k \ge \cdots \ge i_1 \ge n} b^{(k+1)}_{(i_1 \cdots i_k i_{k+1})} \Big \|_p
 \le\Big( \sum_{i_{k+1} \ge i_k \ge \cdots \ge i_1 \ge n} \big\|b^{(k+1)}_{(i_1 \cdots i_k i_{k+1})} \big \|_p^p\Big)^{1/p}\,.
 \end{split}\ee
 This is  just the $p$-norm inequality  for $p\le 1$. On the other hand, since the right supports of the $b^{(k+1)}_{(i_1 \cdots i_k i_{k+1})}$'s are pairwise disjoint, the sum on the left hand side is 1-unconditional; thus by the type $p$ property of $L_p(\M)$ we deduce the claim  for $1<p<2$.
 
 Let $1/r=1/p-1/2$. Then by H\"older's inequality, we have
 \be\begin{split}
\Big( \sum_{i_{k+1} \ge i_k \ge \cdots \ge i_1 \ge n} \big\|b^{(k+1)}_{(i_1 \cdots i_k i_{k+1})} \big \|_p^p\Big)^{1/p}
 \le &\Big( \sum_{i_{k+1} \ge i_k \ge \cdots \ge i_1 \ge n} \big\|b^{(k+1)}_{(i_1 \cdots i_k i_{k+1})} \big \|_2^2\Big)^{1/2}\\
&\cdot \Big( \sum_{i_{k+1} \ge i_k \ge \cdots \ge i_1 \ge n} \tau\big(e^{(k+1)}_{(i_1 \cdots i_k i_{k+1})}\big)\Big)^{1/r}\,.
 \end{split}\ee
   However, by \eqref{ortho6} 
     \be\begin{split}
    \sum_{i_{k+1} \ge i_k \ge \cdots \ge i_1 \ge n} \tau\big(e^{(k+1)}_{(i_1 \cdots i_k i_{k+1})}\big)
   \le \frac{2}{\l^{2(k+1)}}\,  \sum_{i_k \ge \cdots \ge i_1 \ge n} \big\| b^{(k)}_{(i_1\cdots i_k )} \big\|^2_2
     \end{split}\ee
 and by \eqref{ortho6b} 
   \be\begin{split}
 \sum_{i_{k+1} \ge i_k \ge \cdots \ge i_1 \ge n} \big\|b^{(k+1)}_{(i_1 \cdots i_k i_{k+1})} \big \|_2^2
 &= \sum_{i_k \ge \cdots \ge i_1 \ge n} \big\|b^{(k+1)}_{(i_1 \cdots i_k)} \big \|_2^2\\
 &\le 4 \sum_{i_k \ge \cdots \ge i_1 \ge n} \big\|b^{(k)}_{(i_1 \cdots i_k)} \big \|_2^2\,.
    \end{split}\ee
 Combining the previous inequalities with \eqref{ortho6c}, we get
  \[
  \Big\| \sum_{i_{k+1} \ge i_k \ge \cdots \ge i_1 \ge n} b^{(k+1)}_{(i_1 \cdots i_k i_{k+1})} \Big \|_p
  \le 2^{1/p+1/2}\l^{1-2/p}\, 4^{k/p}\l^{k(1-2/p)}\,\|b\|_2^{2/p}\,.
  \]
 Recalling that $\l^{2-p}>4$, we deduce 
 \[
\lim_{k\to\8}\Big\| \sum_{i_{k+1} \ge i_k \ge \cdots \ge i_1 \ge n} b^{(k+1)}_{(i_1 \cdots i_k i_{k+1})} \Big \|_p=0,
\]
  as desired.
 
 \medskip\noindent{\it Step $5$}. We are now in a position to end the proof of the theorem. Define
  $$\lambda_{(0)} = \sqrt{3} \,\l\, \tau (e)^{1/p}\;\text{ and }\; a_{(0)} = \frac{1}{\lambda_{(0)}} \,g.$$
 Letting $k\to\8$ in \eqref{ortho5} and using \eqref {ortho5b}, we conclude that
\be\begin{split}
 b 
 &= g + \sum_{i_1 \ge n} g_{(i_1)} + \cdots + \sum_{i_k \ge \cdots \ge i_1 \ge n} g_{(i_1 \cdots i_k)} + \cdots\\
&=\lambda_{(0)} a_{(0)} +  \sum^{\infty}_{k=1}\, \sum_{i_k \ge \cdots \ge i_1 \ge n} \lambda_{(i_1 \cdots i_k)} a_{(i_1 \cdots i_k)}\\
&=\sum^{\infty}_{k=0}\, \sum_{i_k \ge \cdots \ge i_1 \ge n} \lambda_{(i_1 \cdots i_k)} a_{(i_1 \cdots i_k)}\,.
\end{split}\ee
holds in $L_2 (\M)$. By \eqref{ortho5b}, \eqref{ortho6} and \eqref{ortho6b}, one has
\be\begin{split}
\sum^{\infty}_{k=1}\, \sum_{i_k \ge \cdots \ge i_1 \ge n} |\lambda_{(i_1 \cdots i_k)}|^p 
 &= 3^{p/2} \sum^{\infty}_{k=1} \l^{p(k+1)} \sum_{i_k \ge \cdots \ge i_1 \ge n} \tau (e^{(k)}_{(i_1 \cdots i_k)})\\
 &= 3^{p/2} \sum^{\infty}_{k=1}\l^{p(k+1)} \sum_{i_{k-1} \ge \cdots \ge i_1 \ge n} \tau (e^{(k)}_{(i_1 \cdots i_{k-1})})\\
&\le2 \cdot 3^{p/2} \l^p\, \sum^{\infty}_{k=1} \frac{1}{\l^{k(2-p)}} \,\sum_{i_{k-1} \ge \cdots \ge i_1 \ge n} \| b^{(k-1)}_{(i_1 \cdots i_{k-1})} \|^2_2\\
 &\le \frac{3^{p/2} \l^p}2\, \| b \|^2_2\, \sum^{\infty}_{k=1} \frac{4^k}{\l^{k(2-p)}}\\
  &\le\frac{2\cdot 3^{p/2} \l^p}{\l^{2-p}-4}\,\tau (e)\,.
\end{split}\ee
Thus
\be\begin{split}
\sum^{\infty}_{k=0}\, \sum_{i_k \ge \cdots \ge i_1 \ge n} |\lambda_{(i_1 \cdots i_k)}|^p 
\le 3^{p/2}\,\frac{\l^2-2\l^p}{\l^{2-p}-4}\,\tau (e)\,.
\end{split}\ee
Finally, we get the desired decomposition of $a$ into $(p, \infty)_c$-atoms:
\be
a = \tau (e)^{-1/p} b = \sum^{\infty}_{k=0}\, \sum_{i_k \ge \cdots \ge i_1 \ge n} \tau (e)^{-1/p} \,\lambda_{(i_1 \cdots i_k)} a_{(i_1 \cdots i_k)} 
\ee
such that
\be
\sum^{\infty}_{k=0}\, \sum_{i_k \ge \cdots \ge i_1 \ge n} \big|\tau (e)^{-1}\,\lambda_{(i_1 \cdots i_k)}\big|^p
 \le3^{p/2}\,\frac{\l^2-2\l^p}{\l^{2-p}-4}\,.
\ee
This completes the proof.
\end{proof}

Combining Theorem \ref{th:infty-atomicDecomp} with \cite[Theorem 2.4]{Bekjan-Chen-Perrin-Y} (cf. Corollary \ref{atom-1}) yields the following corollary, which improves the corresponding result of \cite{Hong-Mei}.

\begin{corollary}\label{atom-infty}
We have
\[
\h_1^c(\M)=\h_{1, \at_\8}^c(\M)
\]
with equivalent norms. 
 \end{corollary}

We also have the following atomic decomposition of $\h_p^c(\M)$ for $1<p<2$ 

\begin{corollary}
Let $1<p<2$. Then any $x \in \h_p^c(\M)$ admits a decomposition of the form
 $$x=\sum_{k=1}^\8 \l_k a_k\,,$$
where for each $k,$ $a_k$ is a $(p, \8)_c$-atom or an element in the unit ball of $L_p (\M_1),$ and $\lambda_k \in \mathbb{C}$ satisfying 
 $$\sum_k |\lambda_k|^p \le C_p\|x\|_{\h_p^c}^p\,,$$
where $C_p$ is a positive constant depending only on $p$.
   \end{corollary}

\begin{proof} 
This immediately follows from Theorem~\ref{main}, Remark~\ref{alg-crude}, Proposition~\ref{crude}, and Theorem~\ref{th:infty-atomicDecomp}.
\end{proof}


\section{Applications}\label{Applications}


We give some applications of the previous results.

\subsection{Dual space of $\mathbf{\h_p^c(\M)}$ when $\mathbf{0<p<1}$}
In this subsection, we will discuss a problem raised in \cite{Bekjan-Chen-Perrin-Y} about the characterization of  the dual space of the Hardy space $\h_p^c(\M)$  when $0<p<1$.  Recall that for $1\leq p <\infty$, the dual spaces of the Banach spaces $\h_p^c(\M)$ and $\h_p(\M)$ are well-understood. We refer to \cite{JX,Perrin} for details. For $0<p<1$, a description of the dual space of $\h_p^c(\M)$  was provided in \cite[Theorem~3.3]{Bekjan-Chen-Perrin-Y}. However, using the  commutative setting as a guide (see \cite[Theorem~2.24]{Weisz2}), it is  desirable to have a description  of such  dual space as Lipschitz space. We explore this below. First, we review the noncommutative Lipschitz space and discuss its connection with atomic decomposition.

For $\b\ge 0$, we recall  the column Lipschitz space of order $\b$  defined by
 \[
 \La_\b^c(\M)=\left\{x\in L_2(\M)\;:\; \|x\|_{ \La_\b^c}<\infty\right\},
 \]
where
\[
 \|x\|_{ \La_\b^c}=\max\Big\{\|x_1\|_\infty,\;\; \sup_{n\ge1}\sup_{e\in\mathcal{P}_n}\frac{\|(x-x_n)e\|_2}{\T(e)^{\b+1/2}}\Big\}
 \]
 with  $\mathcal{P}_n$ denotes the lattice projections of $\M_n$.
Note that when $\b=0$,  we recover the column \lq\lq little" $\mathsf{bmo}$-space $\mathsf{bmo}^c(\M)$ (see \cite{Bekjan-Chen-Perrin-Y}). Motivated by the noncommutative John-Nirenberg inequality of \cite{Hong-Mei}  and the atomic decomposition in the previous sections, we introduce the following more general Lipschitz spaces. 

Let additionally $1\leq \gamma < \infty$ . Define 
 \[
 \La_{\b,\g}^c(\M)=\left\{x\in L_2(\M)\;:\; \|x\|_{ \La_{\b,\g}^c}<\infty\right\},
 \]
where
\[
 \|x\|_{ \La_{\b,\g}^c}=\max\Big\{\|x_1\|_\infty,\;\; \sup_{n\ge1}\sup_{e\in\mathcal{P}_n}\frac{\|(x-x_n)e\|_{\h_\g^c}}{\T(e)^{\b+1/\g}}\Big\}\,.
 \]
Note that  $\La_{\b,2}^c(\M)=\La_\b^c(\M)$. It easily follows from H\"older's inequality that if $1\leq \g_1 <\g_2<\infty$ then 
 $\La_{\b,\g_2}^c(\M) \subseteq \La_{\b,\g_1}^c(\M)$ with the inclusion being contractive. 
We will show that the reverse inclusion holds too, so $ \La_{\b,\g}^c(\M)$ is independent of $\g$ (see Corollary~\ref{independence} below)

We also define the subspace:
\[
\La_{\b,\g}^{0,c}(\M)=\Big\{x \in \La_{\b,\g}^c(\M): \E_1(x)=0\Big\}.
\]

Recall that $\h_{p,\,\at_q}^{c}(\M)$ with $0<p\leq 1$ and $1<q<\infty$ is the atomic space defined at the beginning of the previous section. Let $\h_{p,\,\at_q}^{0,c}(\M)$ be its subspace of all $x$ with $\E_1(x)=0$. In the following, $q'$ denotes the conjugate index of $q$.

\begin{proposition}\label{lemma-L}
Let $0<p\leq 1$, $1<q<\infty$, and $\b=1/p-1$. Then 
 \[
 \big(\h_{p,\,\at_q}^{0,c}(\M)\big)^*=\La_{\b,q'}^{0,c}(\M)  \quad \text{with equivalent norms}.
 \]
\end{proposition}
\begin{proof}
 First note that by Theorem~\ref{th:infty-atomicDecomp},   $L_2^0(\M) \subset \h_{p,\,\at_q}^{0,c}(\M)$; moreover, it is easy to see that  $L_2^0(\M)$ is dense in $\h_{p,\,\at_q}^{0,c}(\M)$.
 
 We start to show the inclusion $\La_{\b,q'}^{0,c}(\M) \subseteq(\h_{p,\,\at_q}^{0,c}(\M))^*$.
Let $x \in \La_{\b,q'}^{0,c}(\M)$. If $a$ is a $(p,q)_c$-atom  with  $\E_n(a)=0$ for some  $n\geq 1$ and  $a=ae$ for some projection $e\in \M_n$ satisfying  $\|a\|_{\h_q^c}\leq \T(e)^{1/q-1/p}$, then using the isomorphism  $(\h_q^c(\M))^*= \h_{q'}^c(\M)$ proved in \cite{JX}, we have
\begin{align*}
\big| \T(x^*a)\big| &= \big| \T\big( (x-x_n)^*ae\big)\big|\\
&\leq C_q \big\| (x-x_n)e\big\|_{\h_{q'}^c} \|a\|_{\h_q^c}\\
&\leq C_q \big\|(x-x_n)e\big\|_{\h_{q'}^c} \T(e)^{1/q-1/p}\\
&=C_q\big\|(x-x_n)e\big\|_{\h_{q'}^c} \T(e)^{-\beta-1/q'}\\
&\leq  C_q \|x\|_{\La_{\b,q'}^{0,c}}.
\end{align*}
Thus,  for every $y \in L_2^0(\M)$, the following inequality holds:
\[
\big|\T\big(x^*y\big) \big| \leq C_q \big\|x \big\|_{\La_{\b,q'}^{0,c}} \big\|y\big\|_{\h_{p,\,\at_q}^{0,c}}.
\]
Hence,  the map $\phi_x: y\mapsto  \T(x^*y)$ extends to a continuous functional on $\h_{p,\,\at_q}^{0,c}(\M)$ with norm less than or equal to $C_q\|x\|_{\La_{\b, q'}^{0,c}}$.

Conversely,  let $\phi \in (\h_{p,\,\at_q}^{0,c}(\M))^*$. Since $L_2^0(\M) \subset \h_{p,\,\at_q}^{0,c}(\M)$, there exists $x \in L_2^0(\M)$ such that
\[
\phi(y) =\T(x^*y), \quad y\in L_2^0(\M).
\]
Fix $n\geq 1$ and $e \in \cal{P}_n$.  By duality, we may choose  $y \in \h_q^c(\M)$  with $\|y\|_{h_q^c} \leq C_q'$ so that 
\[
\T(e(x-x_n)^*y) = \big\| (x-x_n)e\big\|_{\h_{q'}^c}.
\]
Clearly, we may assume that $\E_n(y)=0$ and $ye=y$.
Set
\[
 a=\frac{y}{\|y\|_{\h_q^c} \T(e)^{1/p-1/q}}.
\]
Then $a$ is a $(p,q)_c$-atom and
\begin{align*}
\|\phi\| &\geq \big|\T\big( (x-x_n)^*a\big)\big| \\
&= \frac{1}{\|y\|_{\h_q^c}\T(e)^{1/p-1/q}} \T\big( e(x-x_n)^*y\big)\\
&\geq {C_q'}^{-1} \frac{1}{\T(e)^{\b +1/q'}} \big\| (x-x_n)e\big\|_{\h_{q'}^c}
\end{align*}
Taking supremum over $n$ and $e \in \cal{P}_n$, we get $\|\phi\| \geq  C_p'^{-1}\|x\|_{\La_{\b,q'}^{0,c}}$.
\end{proof}

\begin{remark}
 For the special case  $\g=2$,  we have 
 \[
 \big(\h_{p,\,\at}^{0,c}(\M)\big)^*=\La_{\b}^{0,c}(\M)  \quad \text{isometrically}.
 \]
 On the other hand, using the duality $(\h_1^c(\M))^*=\mathsf{bmo}^c(\M)$, we also have
\[
( \h_{p,\mathsf{bmo}}^{0,c})^*= \La_{\b, 1}^{0,c}(\M) \quad \text{with equivalent norm}.
\]
 \end{remark} 

\begin{remark}\label{remark-trivial1}
In general,  one cannot state  Proposition~\ref{lemma-L} for the quasi-Banach space $\h_{p,\at}^c(\M)$ when $0<p<1$.  This is the case since $L_p(\M_1)$ is a complemented subspace of $\h_{p,\at}^c(\M)$ and  $L_p(\M_1)$
has trivial dual if $\M_1$ is not atomic (cf. \cite{Watanabe}).  On the other hand, if $\M_1$ is a type~I atomic von Neumann algebra,  then we  have $(L_p(\M_1))^*$ is isometric to $\M_1$ and therefore we may state  that for $0<p<1$,
\begin{equation}\label{full-dual-1}
 \big(\h_{p,\at}^{c}(\M)\big)^*=\La_\b^{c}(\M)  \quad \text{isometrically}.
\end{equation}
\end{remark}

 It is a natural question to ask if the same  statement holds for $\h_p^c(\M)$.  A positive answer for the particular case of noncommutative  dyadic-martingales  was obtained  recently in \cite[Theorem~1.2]{Jiao-Zhou-Lian-Zanin}.
Our  aim is to show that the weaker form of atomic decomposition stated in Corollary~\ref{weak-decomp} is sufficient  to answer this question positively.  The following result extends   \cite[Theorem~2.6]{Bekjan-Chen-Perrin-Y} to the full range $0<p\leq 1$ and thereby  solves \cite[Problem~4]{Bekjan-Chen-Perrin-Y}.

  \begin{theorem}\label{dual}
  Let $0<p\leq 1$ and $\b=p^{-1}-1$. Then
 \[
 \big(\h_p^{0,c}(\M)\big)^*=\La_\b^{0,c}(\M)  \quad \text{with equivalent norms}.
 \]
 More precisely,  if $y \in \La_\b^{0,c}(\M)$, the map defined by
 \[
 \phi_y: x \mapsto \T(y^*x), \quad x \in L_2^0(\M)
 \]
 extends to a  bounded linear functional on $\h_p^{0,c}(\M)$ satisfying the inequality:
 \[
 \big\| \phi_y \big\|_{(\h_p^{0,c})^*} \leq  \sqrt{2/p}\, \big\|y \big\|_{\La_\b^{0,c}}.
 \]
 Conversely, any  $ \varphi \in \big(\h_p^{0,c}(\M)\big)^*$ is given by the above formula for some $y \in \La_\b^{0,c}(\M)$ satisfying:
 \[
\big\|y \big\|_{\La_\b^{0,c}} \leq  \big\| \phi\big\|_{(\h_p^{0,c})^*}.
 \]
 \end{theorem}
\begin{proof}
Since the formal inclusion $\h_{c,\at}^{0,c}(\M) \subseteq \h_p^{0,c}(\M)$ is a contraction and $(\h_{p,\at}^{0,c}(\M))^*=\La_\b^{0,c}(\M)$,  it is  straightforward to  deduce that if $\varphi \in (\h_p^{0,c}(\M))^*$ then there exists a unique $y \in \La_\beta^{0,c}(\M)$ with $\|y\|_{\La_\b^{0,c}} \leq \|\varphi\|_{(\h_p^{0,c})^*}$ and so that:
\[
\varphi(x)  =\phi_y(x)=\T(xy^*),  \ \  \forall x \in L_2^0(\M).
\]

Conversely, let $y \in \La_\beta^{0,c}(\M)$   and denote by $\phi_y$ the functional induced by $y$  on $\h_{p,\at}^{0,c}(\M)$ according to Proposition~\ref{lemma-L}. We claim  that $\phi_y$ defines a bounded functional on $\h_p^{0,c}(\M)$.

Fix $x \in L_2^0(\M)$. Write $x=\sum_{k\geq 1}\lambda_k a_k$ according to Corollary~\ref{weak-decomp} where the $a_k$'s are $(p,2)_c$-atoms. Then
\[
\T(y^*x)=\sum_{k\geq 1}\lambda_k\T(y^*a_k).
\]
As $\phi_y \in (\h_{p,\at}^{0,c}(\M))^*$, we have
\begin{align*}
|\T(y^*x)| &\leq   \sum_{k\geq 1} |\lambda_k|\, |\phi_y(a_k)|\\
&\leq  \sum_{k\geq 1} |\lambda_k|\, \|y\|_{\La_\b^{0,c}}\\
&\leq  \sqrt{2/p} \,\|y\|_{\La_\b^{0,c}} \|x\|_{\h_p^c}
\end{align*}
where in the last inequality, we use the estimate from Corollary~\ref{weak-decomp} (iv).
 This shows that the  functional $\varphi_y$ extends to a continuous functional on $ \h_p^{0,c}(\M)$ with
 \[
 \|\varphi_y\|_{(\h_p^{0,c})^*} \leq \sqrt{2/p}\, \|y\|_{\La_\b^{0,c}}.
 \]
 The proof is complete.
\end{proof}

\begin{corollary}\label{independence}
 For $\beta\geq 0$ and $1\leq \g<\infty$,
\[
\La_{\b, \g}^{c}(\M)=\La_{\b}^{c}(\M) \quad \text{with equivalent norm}.
\]
\end{corollary}

\begin{proof}
If  $1\leq \g <2$, this is an immediate consequence of  Proposition~\ref{lemma-L}  and Theorem~\ref{th:infty-atomicDecomp}. Assume that $\g>2$. Let $x \in \La_{\b}^c(\M)$ and fix $p$ such that $\b=1/p -1$. By Theorem~\ref{dual}, $x$ induces a functional $\varphi_x$ on $\h_p^{0,c}(\M)$ with $\|\varphi_x\|_{(\h_{p}^{0,c})^*}
\leq \sqrt{2/p}\, \|x\|_{\La_{\b}^{0,c}}$. Since the formal inclusion $\h_{p,\,\at_{\g'}}^{0,c}(\M) \subseteq \h_p^{0,c}(\M)$  is a contraction, we have $\|\varphi_x\|_{(\h_{p,\,\at_{\g'}}^{0,c})^*}
\leq \sqrt{2/p}\, \|x\|_{\La_{\b}^{0,c}}$. By the proposition above,
$\|x\|_{\La_{\b,\g}^{0,c}}
\leq C_\g \|x\|_{\La_{\b}^{0,c}}$ for some constant $C_\g$. On the other hand, since $\g>2$, we already have $\|x\|_{\La_{\b}^{0,c}} \leq \|x\|_{\La_{\b,\g}^{0,c}}$
\end{proof}


\subsection{Fractional integrals on $\mathbf{\h_p(\M)}$  for $\mathbf{0<p<1}$}

 In this subsection, we use the atomic decomposition from Theorem~\ref{main} and Corollary~\ref{weak-decomp} to study  boundedness of fractional  integrals defined on $\h_p^c(\M)$ when $0<p<1$. We first recall the general setup and background for fractional integrals.
 We  further assume that $\M$ is a hyperfinite von Neumann algebra  and the filtration $(\M_n)_{n\geq 1}$ consists of finite dimensional von Neumann subalgebras of $\M$.

For $n\geq 1$,  we change the notation for the difference operator $d_n=\E_n -\E_{n-1}$: now we set $\mathcal{D}_n=\E_n -\E_{n-1}$ (where $\E_0=0$) and 
 \[
 \mathcal{D}_{n,p}:=\mathcal{D}_n(L_p(\M))=\Big\{x \in L_p(\M_k): \E_{n-1}(x)=0\Big\}.
 \]
  Since $\dim(\M_n)<\infty$,   the $\mathcal{D}_{n,p}$'s   are well-defined for all $0<p\leq \infty$. Moreover,    for $p\neq q$,  the two linear spaces  $\mathcal{D}_{n,p}$ and $\mathcal{D}_{n,q}$ coincide as sets. In particular,
 the formal identity $\iota_k: \mathcal{D}_{k,\infty} \to \mathcal{D}_{k,2}$  forms a natural  isomorphism between the two spaces. Following \cite{RW3},  we set for $n\geq 1$,
\begin{equation}\label{zeta}
\zeta_n:=\frac1{\| \iota_n^{-1}\|^{2}}.
\end{equation}
Clearly, $0<\zeta_n \leq 1$ for all $n\geq 1$ and   $\lim_{n\to \infty} \zeta_n=0$. Moreover,  for every  $x \in \mathcal{D}_{n,2}$, we have
\begin{equation}\label{infty-2}
\|x\|_\infty \leq \zeta_n^{-1/2} \|x\|_2.
\end{equation}
Our primary example is  the  standard  filtration on the  hyperfinite  type ${\rm II}_1$-factor $\mathcal{R}$. For this specific case,  we have $\zeta_n=2^n$ for $n\geq 1$ which is identical to the case of  classical dyadic martingales  formulated in \cite{Chao-Ombe}. We consider the following special type of martingale transforms:
\begin{definition}\label{definition:fractional}
For a given noncommutative  martingale $x=(x_n)_{n\geq 1}$ and $\a \in (0,\infty)$, we define the \emph{fractional
integral of order $\alpha$}  of $x$ to be the martingale $I^\alpha x=\{(I^\alpha x)_n\}_{n\geq 1}$ where  for every $n\geq 1$,
\[
(I^\alpha x)_n = \sum_{k=1}^{n} \zeta_k^{\alpha}dx_k
\]
with the sequence of scalars $(\zeta_k)_{k\geq 1}$ from \eqref{zeta}.
\end{definition}
In \cite{RW3},  the notation $I^\a x$ was used only for $0<\a<1$ but we will use here the same notation for the full range of $\a$. Fractional integrals of classical dyadic martingales  were studied in \cite{Chao-Ombe}.  We refer to \cite{RW3} for  an extensive   treatment of the case of noncommutative martingales. The results in \cite{RW3} cover mainly the Banach space range.
Below, we consider the boundedness of fractional integrals  defined on  martingale  Hardy spaces  $\h_p$ for $0<p<1$. The following is the main result for this subsection. It complements results from the appendix section of \cite{RW3}.

\begin{theorem}\label{fractional}
Assume that $0<p \leq 1$, $p<q<2$, and $\a=1/p -1/q$. The fractional integral $I^\a$ extends to  a bounded linear map from $\h_p^c(\M)$ into $\h_q^c(\M)$.

Similarly, $I^\a$ is also bounded from $\h_p(\M)$ into $\h_q(\M)$.
\end{theorem}

The  decisive part of the argument is contained in the next lemma:
\begin{lemma}\label{frac-crude}
Assume that $0<p_0 \leq 1$, $p_0<p_1<2$, and $\g=1/p_0-1/p_1 \in (0,1/2)$. There exists a constant  $C_\g$ such that $C_\g^{-1}I^\g a$ is a $(p_1,2)_c$-crude atom whenever $a$ is  a $(p_0,2)_c$-crude atom.
\end{lemma}
\begin{proof}
Let $a$ be a $(p_0,2)_c$-crude atom and fix $r_0$ so that $1/p_0=1/2 +1/r_0$. There exist $n\geq 1$ and a factorization $a=yb$ with $\E_n(y)=0$, $\|y\|_2=1$, $b\geq 0$, and $b\in L_{r_0}(\M_n)$ with $\|b\|_{r_0}\leq 1$.

Since $\g \in (0,1/2)$,  there exists $r>2$  such that $\g =1/2 -1/r$. According to \cite[Corollary~2.7]{RW3}, $I^\g: L_2(\M) \to L_r(\M)$ is bounded. If we set  $C_\g:=\big\| I^\g: L_2(\M) \to L_r(\M)\big\|$, then  the operator $C_\g^{-1}I^\g y$ belongs to  $L_r(\M)$. Moreover, one can easily check that
$\E_n(C_\g^{-1}I^\g y)=0$ and $\big\|C_\g^{-1}I^\g y\big\|_r \leq 1$.

Let $\widehat{y}=C_\g^{-1} (I^\g y ) b^{\g r_0}$ and $\widehat{b}=b^{1-\g r_0}$. Then $C_\g^{-1} I^\g a= \widehat{y}  . \widehat{b}$ and we claim that this factorization satisfies the definition of $(p_1,2)_c$-crude atom. Indeed,  it is clear that
$\E_n(\widehat{y})=0$. Also since $\|b\|_{r_0}\leq 1$, by H\"older's inequality, it follows that
\begin{align*}
\big\|\widehat{y}\big\|_2 &\leq  \big\|C_\g^{-1} I^\g y\big\|_r  \big\|b^{\g r_0}\big\|_{1/\g}\\
&\leq  \big\|b^{r_0}\big\|_1^\g \leq 1.
\end{align*}
On the other hand, if $1/p_1=1/2 +1/r_1$, then one can easily verify that $r_1(1-\g r_0)=r_0$. Consequently, $\widehat{b} \in L_{r_1}(\M_n)$ with  $\big\|\widehat{b}\big\|_{r_1}^{r_1}= \big\| b \big\|_{r_0}^{r_0} =1$. The lemma is verified.
\end{proof}

\begin{proof}[Proof of Theorem~\ref{fractional}]  $\bullet$ We consider first the boundedness of $I^\a$  for column conditioned Hardy spaces.  We divide the proof into several cases.

-Case~1. Assume that $0<p\leq 1$, $p<q<2$ and $\a=1/p -1/q<1/2$.

Fix $x \in \h_p^c(\M)$ and consider its decomposition
${x=x_1 + \sum_{l\geq 1} \lambda_l a_l}$  according to Theorem~\ref{main} and Remark~\ref{alg-crude}. Here,  $x_1 \in L_p(\M_1)$ and for $l\geq 1$, $a_l$ is a $(p,2)_c$-crude atoms with $\E_l(a_l)=0$. We also have the estimate
 $\sum_{l\geq 1}|\lambda_l | \leq \sqrt{2/p}\big\|x\big\|_{\h_p^c}$.

 Since $0<\a<1/2$,
  Lemma~\ref{frac-crude} applies.  There  exists $C_\alpha$ such that   for every $l\geq 1$, $C_\a^{-1} I^\a (a_l)$ is a $(q,2)_c$-crude atoms  with $\E_l(C_\a^{-1} I^\a (a_l))=0$. It is straightforward to verify that  if $\sigma= \sqrt{2/p}\|x\|_{\h_p^c}$ then  $z=\sigma^{-1} \sum_{l\geq 1} \lambda_l C_\a^{-1} I^\a (a_l)$  is an algebraic $\h_q^c$-atom and
\[
I^\a x= \zeta_1^\a x_1 + C_\a \sigma z.
\]
We note  by assumption that $2p/q>1$.  As $\a q/2= q/(2p)-1/2$,  it follows from \cite[Lemma~2.4]{RW3} that
 $I^{\a q/2}: L_{2p/q}(\M) \to L_2(\M)$ is bounded. This  implies  in particular that for some constant $C_\a'$, we have
 \[
 (\zeta_1^\a \|x_1\|_q)^q= (\zeta_1^{\a q/2} \| |x_1|^{q/2}\|_2)^2  \leq (C_\a' \| |x_1|^{q/2}\|_{2p/q})^2= (C_\a')^2 \|x_1\|_p^q.
 \]
 We can conclude that
\begin{align*}
\big\| I^\a x \big\|_{\h_{q,\alg}^c}  &\leq (C_\a')^{2/q} \big\|x_1\big\|_p + C_\a \sqrt{2/p} \big\| x\big\|_{\h_p^c}\\
&\leq  \big[(C_\a')^{2/q}  + C_\a \sqrt{2/p}\big]  \big\| x\big\|_{\h_p^c}.
\end{align*}
Since $0<q<2$, this shows that $I^\a: \h_p^c(\M) \to \h_q^c(\M)$ is bounded.

\medskip

We remark that  the particular case $p=1$ is entirely covered by Case~1 since in this specific case,  we always have $0<\a<1/2$. Thus, for the remaining cases, we assume that $0<p<1$.

\medskip

-Case~2. Assume that $0<p<q\leq 1$ and $\a=1/p-1/q \geq 1/2$.

Let $n(\a)= \lfloor{ 2\a}\rfloor+1$ where $\lfloor \cdot \rfloor$ denotes  the  greatest integer function and set $\g= \a/ n(\a)$. Clearly, $0<\g<1/2$.  Let $p_0=p$ and for  $1\leq m \leq n(\a)$, we define inductively $p_m$  to be the index that satisfies $\g= 1/p_{m-1}-1/p_m$. We note that  $(p_m)_{m=0}^{n(\a)}$ is an  increasing  finite sequence of indices and
$p_{n(\a)}=q$. From Step~1, for every $1\leq m \leq  n(\a)$, $I^\g: \h_{p_{m-1}}^c(\M) \to \h_{p_m}^c(\M)$ is bounded.
We apply   $I^\g$ successively $n(\a)$-times and get $I^\a$ as the compositions of bounded maps:
\[
I^\a: \h_{p}^c(\M) \xrightarrow{I^\g} \h_{p_1}^c(\M) \xrightarrow{I^\g}  \h_{p_2}^c(\M) \xrightarrow{I^\g} \dots  \xrightarrow{I^\g} \h_{p_{n(\a)-1}}^c(\M)  \xrightarrow{I^\g}  \h_q^c(\M).
\]
Thus, $I^\a: \h_p^c(\M) \to \h_q^c(\M)$ is bounded.

\medskip

-Case~3. $0<p<1<q<2$.

Let $\a_1=1/p -1$ and $\a_2=1-1/q$. By Case~2, $I^{\a_1}: \h_p^c(\M) \to \h_1^c(\M)$ is bounded. Similarly, since $0<\a_2<1/2$, we have from Case~1 that $I^{\a_2}: \h_1^c(\M) \to \h_q^c(\M)$ is bounded. It follows that $I^\a= I^{\a_2} I^{\a_1}:\h_p^c(\M) \to \h_q^c(\M)$ is bounded. This completes the proof for the column part. By taking adjoints, we also have  that $I^\a:\h_p^r(\M) \to \h_q^r(\M)$ is also bounded.

\medskip

$\bullet$ We now verify that $I^\a:\h_p^d(\M) \to \h_q^d(\M)$. This will be deduced from the following inequality:
\begin{equation}\label{diag}
\zeta_k^{\a q} \big\|dx_k \big\|_q^q \leq \big\|dx_k\big\|_p^q, \quad k\geq 1.
\end{equation}
A version of this  inequality was stated in the proof of \cite[Corollary~A.7]{RW3} but the argument given there contains an error. We include  the corrected argument.
Fix $k\geq 1$.  First, we  claim that  $\|dx_k\|_\infty \leq  \zeta^{-1/p} \|dx_k\|_p$. Indeed, by  the definition of $\zeta_k$,  we have $\|dx_k\|_\infty \leq \zeta_k^{-1/2} \|dx_k\|_2$. Since
$\|dx_k\|_2 \leq \|dx_k\|_\infty^{1-(p/2)} \|dx_k\|_p^{p/2}$, it follows that
$\|dx_k\|_\infty \leq \zeta_k^{-1/2} \|dx_k\|_\infty^{1-(p/2)} \|dx_k\|_p^{p/2}$ which implies the claim. We now have the following estimates:
\begin{align*}
\big\|dx_k\big\|_q^q \leq &\big\|dx_k\big\|_\infty^{q-p} \big\|dx_k\big\|_p^p\\
&\leq \big(\zeta_k^{-1/p} \big\|dx_k\big\|_p\big)^{q-p} \big\|dx_k\big\|_p^p\\
&=\zeta_k^{-(q-p)/p} \big\|dx_k\big\|_p^q.
\end{align*}
As $\a q=(q-p)/p$, we have verified \eqref{diag} which  in particular implies that
\[
\big\| I^\a:\h_p^d(\M) \to \h_q^d(\M)\big\| \leq 1.
\]
Combining the column, the row, and the diagonal versions, we obtain the second statement of the theorem.
\end{proof}

A natural question that arises from  Theorem~\ref{fractional} is wether the boundedness of fractional integrals  remains valid when the domain is the   Hardy space $\H_p^c(\M)$  for $0<p<1$. We consider below a special situation where this is the case.  We  recall the notion of regular filtration.

\begin{definition}
A filtration  $(\M_n)_{n\geq 1}$ is  called  \emph{regular} with constant $C$ (or $C$-regular for short) if  for every  positive $x \in L_1(\M)$ and $n \geq 1$, the following holds:
\[
\E_n(x) \leq C\E_{n-1}(x).
\]
\end{definition}
Examples of  regular filtrations are the noncommutative dyadic filtration  and more generally  filtrations associated with bounded Vilenkin groups on the hyperfinite type ${\rm II}_1$-factor $\cal{R}$ (see \cite[Lemma~2.2]{Wu1} and \cite[Lemma~3.3]{Scheckter-Sukochev}, respectively).
In  classical martingale theory, it is a well-known fact that the two Hardy spaces $\h_p$ and $\H_p$ coincide for all $0<p < \infty$ whenever the filtration is regular (\cite[Corollary~2.23]{Weisz}). For the noncommutative case,  it is easy to deduce from the definition of regularity and the dual Doob inequality (\cite{Ju}) that  for $2\leq p<\infty$, the two Hardy spaces
$\h_p^c(\M)$  and $\H_p^c(\M)$ coincide when the filtration is regular.
Our  next result shows that this fact remains valid for $0<p<2$. This  may be of independent interest.
\begin{theorem}\label{reg-hardy} Assume that the  filtration  $(\M_n)_{n\geq 1}$ is $C$-regular for some $C>0$. Then  for $0<p<2$,
\[
\h_p^c(\M)=\H_p^c(\M) \quad \text{with equivalent norms.}
\]
More precisely, if $x \in \H_p^c(\M)$,  then
\[
\max\{\sqrt{p/2}, \sqrt{1/C}\,\}\big\|x \big\|_{\H_p^c} \leq \big\|x \big\|_{\h_p^c} \leq  C^{1/p-1/2} \big(\frac{2}{p}\big)^{1/p}  \big\| x \big\|_{\H_p^c}.
\]
\end{theorem}
\begin{proof} It was  proved in  \cite[Theorem~4.11]{Jiao-Ran-Wu-Zhou} that for general filtration, $\| x\|_{\H_p^c} \leq \sqrt{2/p} \|x\|_{\h_p^c}$.
On the other hand, it follows from the definition of $C$-regularity that
$S_c^2(x) \leq C s_c^2(x)$ which clearly implies that $\|x\|_{\H_p^c} \leq \sqrt{C} \|x\|_{\h_p^c}$. Thus, we have the first inequality.

For  the second inequality, we may assume by approximation that $x \in \M$ and for every $n\geq 1$, $S_{c,n}(x)$ is invertible with bounded inverse. Since $s_{c,1}(x)=S_{c,1}(x)$, it follows that the $s_{c,n}(x)$'s are also  invertible with bounded inverses. First, we put the $\h_p^c$-norm of $x$ in  the following form:
\begin{align*}
\|x\|_{\h_p^c}^p &=\T\big(s_c^p(x)\big)\\
&=\T\big( s_c^{p-2}(x)s_c^2(x)\big)\\
&=\sum_{n\geq 1}\T\big( s_c^{p-2}(x)(s_{c,n}^2(x)-s_{c,n-1}^2(x))\big).
\end{align*}
Since for $n\geq 1$, $s_{c,n}^2(x) \leq s_c^2(x)$ and $0<p<2$, we have
$s_{c}^{p-2}(x) \leq s_{c,n}^{p-2}(x)$. A fortiori,
\begin{align*}
\|x\|_{\h_p^c}^p  &\leq \sum_{n\geq 1}\T\big( s_{c,n}^{p-2}(x)(s_{c,n}^2(x)-s_{c,n-1}^2(x))\big)\\
&=\sum_{n\geq 1} \T\big( s_{c,n}^{p-2}(x) \E_{n-1}(|dx_n|^2)\big)\\
&=\sum_{n\geq 1} \T\big( s_{c,n}^{p-2}(x) |dx_n|^2\big)
\end{align*}
where in the last equality, we use the fact that the sequence $(s_{c,n}(x))_{n\geq 1}$ is  predictable.
The $C$-regularity implies that for every $n\geq 1$, $S_{c,n}^2(x) \leq Cs_{c,n}^2(x)$. Therefore,  $s_{c,n}^{p-2}(x) \leq C^{1-p/2} S_{c,n}^{p-2}(x)$. This further implies that
\begin{align*}
\|x\|_{\h_p^c}^p &\leq C^{1-p/2} \sum_{n\geq 1} \T\big( S_{c,n}^{p-2} (x)|dx_n|^2\big)\\
&=C^{1-p/2}\sum_{n\geq 1} \T\big( S_{c,n}^{p-2}(x)(S_{c,n}^2(x)- S_{c,n-1}^2(x))\big).
\end{align*}
We can now conclude from Lemma~\ref{mart-2} that
\begin{align*}
\|x\|_{\h_p^c}^p &\leq C^{1-p/2} \big(\frac{2}{p}\big)\sum_{n\geq 1} \T\big( S_{c,n}^p(x)- S_{c,n-1}^p(x)\big)\\
&=C^{1-p/2} \big(\frac{2}{p}\big) \big\|x \big\|_{\H_p^c}^p.
\end{align*}
The proof is complete.
\end{proof}
\begin{remark}
Using the row version of Theorem~\ref{reg-hardy}, we may state  the noncommutative extension of \cite[Corollary~2.23]{Weisz} that for regular filtration, $\h_p=\H_p$ for all $0<p<\infty$.
\end{remark}
We can now state from combining Theorem~\ref{fractional} and Theorem~\ref{reg-hardy} that  for the case of regular filtration,  the fractional integral $I^\a: \H_p^c(\M) \to \H_q^c(\M)$ is bounded whenever $0<p<1$, $p<q<2$, and $\a=1/p -1/q$.
In particular, this
 extends \cite[Theorem~3(i)]{Chao-Ombe} to noncommutative  dyadic martingales. It is  still  an open problem if this statement applies to general filtrations.

\
 
 {\it Acknowledgments.}\;  Z. Chen was partially supported by the Natural Science Foundation of China (No.11871468), Q. Xu by  the French ANR project (No. ANR-19-CE40-0002-01).
 


\def\cprime{$'$}
\providecommand{\bysame}{\leavevmode\hbox to3em{\hrulefill}\thinspace}
\providecommand{\MR}{\relax\ifhmode\unskip\space\fi MR }
\providecommand{\MRhref}[2]{%
  \href{http://www.ams.org/mathscinet-getitem?mr=#1}{#2}
}
\providecommand{\href}[2]{#2}

\end{document}